\documentclass[leqno,11pt]{amsart}
\newtheorem{thm}{Theorem}[section]
\newtheorem{cor}[thm]{Corollary}
\newtheorem{lemma}[thm]{Lemma}
\newtheorem{prop}[thm]{Proposition}

\theoremstyle{definition}

\theoremstyle{remark}

\newtheorem{remark}[thm]{Remark}
\newtheorem{notation}[thm]{Notation}

\newcommand{\R}{\mathbb R}

\newcommand{\C}{\mathbb C}

\newcommand{\Z}{\mathbb Z}
\newcommand{\N}{\mathbb N}

\usepackage{mathrsfs}

\newlength{\intwidth}

\title[Accumulation points of recurrence coefficients]{Orthogonal polynomials on several intervals:\\
accumulation points of recurrence coefficients and of zeros}
\author[F. Peherstorfer]{F. Peherstorfer$^1$}
\thanks{This work was supported
by the Austrian Science Fund FWF, project no. P20413-N18}
\thanks{$^1$The last modifications and corrections of this manuscript
were done by the author in the two months preceding this passing away
in November 2009. The manuscript is not published elsewhere 
(submitted by P. Yuditskii and I. Moale).}

\begin{document}

\begin{abstract}
Let $E = \cup_{j = 1}^l [a_{2j-1},a_{2j}],$ $a_1 < a_2 < ... < a_{2l},$ $l \geq 2$ and set
$\mbox{\boldmath$\omega$}(\infty) =
(\omega_1(\infty),...,\omega_{l-1}(\infty))$, where $\omega_j(\infty)$ is the
harmonic measure of $[a_{2 j - 1}, a_{2 j}]$ at infinity. Let $\mu$ be a measure which is
on $E$ absolutely continuous
and satisfies Szeg\H{o}'s-condition and has at most a finite number of point measures outside $E$,
and denote by $(P_n)$ and $({\mathcal Q}_n)$ the orthonormal polynomials and their associated
Weyl solutions with respect to $d \mu$, satisfying the recurrence relation
$\sqrt{\lambda_{2 + n}} y_{1 + n} = (x - \alpha_{1 + n}) y_n -
\sqrt{\lambda_{1 + n}} y_{-1 + n}$. We show that the recurrence coefficients have
topologically the same convergence behavior as the sequence
$(n \mbox{\boldmath$\omega$}(\infty))_{n\in \mathbb N}$ modulo $1$;
More precisely, putting
$( \mbox{\boldmath$\alpha$}^{l-1}_{1 + n},
\mbox{\boldmath$\lambda$}^{l-1}_{2 + n} ) = $ $( \alpha_{[\frac{l-1}{2}]+1+n},...,$
$\alpha_{1+n},...,$ $\alpha_{-[\frac{l-2}{2}]+1+n},$ $\lambda_{[\frac{l-2}{2}]+2+n},$
$...,\lambda_{2+n},$ $...,$ $\lambda_{-[\frac{l-1}{2}]+2+n} )$
we prove that $( \mbox{\boldmath$\alpha$}^{l-1}_{1 + n_\nu},$
$\mbox{\boldmath$\lambda$}^{l-1}_{2 + n_\nu} )_{\nu \in \mathbb N}$ converges
if and only if $(n_\nu \mbox{\boldmath$\omega$}(\infty))_{\nu \in \mathbb N}$ converges modulo $1$
and we give an explicit homeomorphism between the sets of accumulation points of
$( \mbox{\boldmath$\alpha$}^{l-1}_{1 + n},
\mbox{\boldmath$\lambda$}^{l-1}_{2 + n} )$ and $(n\mbox{\boldmath$\omega$}(\infty))$
modulo $1$. As one of the consequences there is a homeomorphism from the so-called
gaps $\sf{X}_{j=1}^{l-1}$ $([a_{2j},a_{2j+1}]^{+} \cup [a_{2j},a_{2j+1}]^{-})$
on the Riemann surface $y^2 = \prod_{j = 1}^{2 l}(x - a_j)$ into the set of accumulation
points of the sequence
$( \mbox{\boldmath$\alpha$}^{l-1}_{1 + n}, \mbox{\boldmath$\lambda$}^{l-1}_{2 + n} )$
if the harmonic measures $\omega_1(\infty),...,\omega_{l-1}(\infty),$ $1$
are linearly independent over the rational
numbers $\mathbb Q.$ Furthermore it is demonstrated, loosely speaking, that the convergence
behavior
of the sequence of recurrence coefficients
$( \mbox{\boldmath$\alpha$}^{l-1}_{1 + n},$
$\mbox{\boldmath$\lambda$}^{l-1}_{2 + n} )$
and of the sequence of zeros of the orthonormal polynomials and
Weyl solutions outside the spectrum is
topologically the same. The above results are proved by deriving first corresponding
statements for the accumulation points of the vector of moments of the
diagonal Green's functions, that is, of the sequence
$( \int x P^2_{n} d \mu, ..., \int x^{l-1} P^2_{n} d \mu,$
$\sqrt{\lambda_{2 + n}}$ $ \int x P_{1 + n} P_{n} d \mu,$
$..., $ $\sqrt{\lambda_{2 + n}}\int x^{l-1}
P_{1 + n} P_{n} d \mu)_{n \in \mathbb N}.$
\end{abstract}

\maketitle

\section{Introduction}

Let $l \in \mathbb N$ and $a_k \in \mathbb R$ for $k = 1,...,2l, a_1 < a_2 < ... < a_{2l}.$
Put $$E = \bigcup\limits_{k = 1}^l E_k, {\rm \ where \ } E_k = [a_{2k-1}, a_{2k}], {\rm \ and \ }
H(x) = \prod\limits_{j=1}^{2l} (x - a_j)$$ and henceforth
let us choose that branch of $\sqrt{H}$ for which $\sqrt{H(x)} > 0$ for $ x \in (a_{2l},\infty).$
For convenience we set
\begin{equation}\label{t1}
    h(x) :=
    \left\{
    \begin{split}
        & \frac{(\sqrt{H})^{+}(x) - (\sqrt{H})^{-}(x)}{2 i} =
          (-1)^{l-k} \sqrt{|H(x)|} {\rm \ for \ } x \in E_k\\
        & 0 {\rm \quad elsewhere }
    \end{split}
    \right.
\end{equation}
where, as usual, $f^{\pm}(x)$ denote the limiting values from the upper and lower
half plane, respectively. Note that $(\sqrt{H})^{-}(x) = - (\sqrt{H})^{+}(x)$ for
$x \in E.$ By $\phi(z,z_0)$ we denote a so-called complex Green's
function for $\overline{\mathbb C} \backslash E$ uniquely determined up to a
multiplication constant of absolute value one (chosen conveniently below), that is,
$\phi(z,z_0)$ is a multiple valued function which is analytic on $\overline{\mathbb C}
\backslash E$ up to a simple pole at $z = z_0,$ has no zeros on $\overline{\mathbb C}
\backslash E$ and satisfies $|\phi(z,z_0)| \to 1$ for $z \to x \in E$
or in other words $\log |\phi(z,z_0)|$ is the (potential theoretic) Green's function with pole at $z = z_0 \in
\overline{\mathbb C} \backslash E$, as usual denoted by $g(z,z_0).$ In the case under
consideration, as it is known \cite{PehIMNR, Wid}, a complex Green's function may be represented as
\begin{equation}\label{t2}
    \phi(z) := \phi(z,\infty) = \exp \left( \int_{a_{2l}}^{z} r_{\infty}(x) \frac{dx}{\sqrt{H(x)}} \right)
\end{equation}
where $r_\infty$ is the unique monic polynomial of degree $l-1$ such that
\begin{equation}\label{t3}
    \int_{a_{2j}}^{a_{2j+1}} r_\infty(x) \frac{dx}{\sqrt{H(x)}} = 0 {\rm \ for \ } j = 0,...,l-1;
\end{equation}
hence
\begin{equation}\label{t4}
    r_\infty(x) = \prod\limits_{j=1}^{l-1} (x - c_j), \ \ c_j \in (a_{2j},a_{2j+1}), j = 1,...,l-1.
\end{equation}
Recall that the so-called capacity of $E$ is given by
\begin{equation}\label{C}
   cap (E) = \lim\limits_{z \to \infty} \left| \frac{z}{\phi(z,\infty)} \right|
\end{equation}
The density of the equilibrium distribution of $E,$ denoted by $\rho,$ becomes
\begin{equation}\label{t5}
   \rho(x) = \frac{1}{2 \pi i} \left( (\frac{\phi^{'}}{\phi} )^+(x) -
   (\frac{\phi^{'}}{\phi}) ^- (x) \right)= \frac{r_\infty(x)}{\pi h(x)}.
\end{equation}

By $\omega(z, B, \overline{\mathbb C} \backslash E)$ we denote the harmonic measure
of $B \subseteq E$ with respect to $\overline{\mathbb C} \backslash E$ at $z,$
which is that harmonic and bounded function on $\overline{\mathbb C} \backslash E_l$
which satisfies for $\xi \in E_l$ that
$\lim_{z \to \xi} \omega(z,B,\overline{\mathbb C} \backslash E_l) = i_B(\xi),$
where $i_B$ denotes the characteristic function of $B.$ For abbreviation we put
\begin{equation*}
    \omega(z, E_k;\overline{\mathbb C} \backslash E) = \omega_k(z).
\end{equation*}
Recall that, $k = 1,...,l,$
\begin{equation*}
    \omega_k(\infty) = \int_{a_{2 k - 1}}^{a_{2 k}} \rho(x) dx.
\end{equation*}

In the following we say that the function $w$ is on $E$ from the Szeg\H{o}-class,
written $w \in {\rm Sz}(E),$ if
\begin{equation}\label{t6}
    \int_E \log \left( w(x) \right) \rho(x) dx > - \infty.
\end{equation}
In this paper we consider positive measures of the form
\begin{equation}\label{7_2}
   d\mu(x) = w(x) dx + \sum\limits_{j=1}^m \mu_j \delta_{d_j}(x)
\end{equation}
where $w \in Sz(E),$ the mass points $d_j,$ $j = 1,...,m,$ are from $\mathbb R \backslash E$
and mutually disjoint and $\mu_j \in {\mathbb R}^{+}.$
By $P_n$ we denote the polynomial of degree $n$ orthonormal with respect to $w$ i.e.,
$$ \int_E P_m(x) P_n(x) d\mu(x) = \delta_{m,n}$$
and the monic orthogonal polynomial are denoted by $p_n(x).$ It is well known that the monic
orthogonal polynomials satisfy a three term recurrence relation
\begin{equation}\label{62}
   p_n(x) = (x - \alpha_n) p_{n-1}(x) - \lambda_n p_{n-2}(x)
\end{equation}
with $p_{-1}(x) := 0$ and $p_0(x) = 1.$
For measures of the form \eqref{7_2} it follows, see \cite[Cor. 6.1]{PehYud}
(in fact more general sets $E$ and measures are considered there),
that the recurrence coefficients are almost periodic in the limit,
more precisely, that there exist real analytic functions
$\alpha$ and $\lambda$ on the torus $[0,1]^{l-1}$ and a constant
$\mbox{\boldmath$c$} \in \R^{l-1}$,
depending on the measure $\mu$, such that
\begin{equation}
\label{F1}
\alpha _{n+1} = \alpha (n \mbox{\boldmath$\omega$}(\infty) -
\mbox{\boldmath$c$})+o(1) \quad \text{and} \quad \lambda _{n+2} =
\lambda (n \mbox{\boldmath$\omega$}(\infty) -
\mbox{\boldmath$c$})+o(1),
\end{equation}
where $\mbox{\boldmath$\omega$}(\infty) = (\omega_1(\infty),...,\omega_{l-1}(\infty))$;
for more details see the remark
following Proposition \ref{prop1} below.
By \eqref{F1} one gets a first, quite good view into
the convergence behavior of the $\alpha _n$'s and $\lambda _n$'s, but by far
not a complete one, since the functions $\alpha$ and $\lambda$
and their mapping properties are not known. The more it is not clear
whether the recurrence coefficients and $(\mbox{\boldmath$\omega$}(\infty))_{n\in \N}$
modulo $1$ have the same topological convergence behavior; that is, that subsequences
converge simultaneously  and that there is a homeomorphism between the set of
accumulation points. The goal of this paper is to show that, under a proper point
of view, there is such a topological equivalence; also between the recurrence
coefficients and the zeros of the polynomials and Weyl solutions outside the
spectrum. As an immediate consequence it follows that the problem of convergence
of the recurrence coefficients is topologically equivalent to the classical
problem in number theory
to find diophantic approximations of the harmonic measures $\mbox{\boldmath$\omega$}(\infty)$,
\cite[V. Kapitel]{Per}, \cite{Hla, Kok}.




Instead of continuing to study the accumulation points
of the recurrence coefficients with the help of \eqref{F1}, which looks rather hopeless,
we may also investigate that ones of the sequences of moments of the diagonal Green's
functions, that is, $ (\int x P^2_n, \sqrt{\lambda_{n+2}}
\int x P_{n + 1} P_n,\int x^2 P^2_n,\sqrt{\lambda_{n+2}}
\int x^2 P_{n + 1} P_n,...,$ $ \int x^m P^2_n,\sqrt{\lambda_{n+2}}$
$ \int x^m P_{n + 1} P_n, ...)_{n \in \mathbb N} $ since it can be shown
(see Proposition \ref{propP1}) that there is a unique correspondence between
such vectors of moments and the vector of recurrence coefficients
$(...,\alpha_{[\frac{m-1}{2}] + 1 + n},\lambda_{[\frac{m-2}{2}] + 2 + n},...,$
$\alpha_{1+n},$ $\lambda_{2+n},...,$ $\lambda_{-[\frac{m-1}{2}]+2+n},
\alpha_{-[\frac{m-2}{2}] + 1 + n},$ $...).$ In fact it even suffices to consider
the truncated vector sequence $(\int x P^2_n,$ $ \sqrt{\lambda_{n+2}}\int x P_{n+1}
P_n,$ $...,$ $\int x^{l-1} P^2_n,$ \\ $ \sqrt{\lambda_{n+2}}$ $\int x^{l-1} P_{n+1} P_n )
_{n \in \mathbb N}$ since, by Corollary \ref{corR31}, it carries all information
already. We show (Theorem \ref{thm4.1}; concerning
the recurrence coefficients see Theorem \ref{5.4new}) that there is a homeomorphism
between
the set of accumulation
points of the truncated vector sequence of moments (respectively, of recurrence
coefficients) and of the sequence $(n \mbox{\boldmath$\omega$}(\infty))$ modulo $1$
and that convergence holds for the same subsequences. In particular this implies that
the set of accumulation points of the truncated vector sequence of moments
is homeomorph to a $l-1$ dimensional
torus if the harmonic measures are linearly independent over $\mathbb{Q}$, where the
homeomorphism is given explicitly even.

Also of special interest is the fact that there is a unique correspondence between the
accumulation points of the discussed truncated vector sequence of moments and the
accumulation points of zeros of the orthogonal polynomials and Weyl solutions outside
the spectrum, see Theorem \ref{thm4.2}.

Next let us represent the weight function $w$ in the form
\begin{equation}\label{xxx}
     w(x) = \frac{w_0(x)}{\rho(x)}
\end{equation}
and let ${\mathcal W}_o(z)$ be that function which is analytic on
$\tilde{\Omega} := \overline {\mathbb C} \backslash E,$ has no zeros
and poles there, is normalized by ${\mathcal W} (\infty) > 0$ and
satisfies the boundary condition
\begin{equation}\label{t7}
   w_0(x) = 1 / {\mathcal W}^{+}_0(x) {\mathcal W}^{-}_0(x).
\end{equation}
For $\mu$ of the form \eqref{7_2} with $w \in {\rm Sz}(E)$ it has been shown
that $P_n/\phi^n$ is uniformly asymptotically equivalent
to the solution of a certain extremal problem, more precisely, that
on compact subsets of
$\Omega := \tilde{\Omega} \backslash \{ d_1,...,d_m \}$
\begin{equation}\label{t8}
    \frac{P_n(z)}{\phi^n(z)} \sim \frac{U {\mathcal W}_0(z) \phi^{'}(z)}
    {\prod\limits_{j=1}^{m} \phi(z;d_j)}
          \left( \frac{
          \prod\limits_{j=1}^{l-1} \phi(z;c_j)
          \prod\limits_{j=1}^{l-1} \phi(z;x_{j,n})^{\delta_{j,n}}
          \prod\limits_{j=1}^{l-1} (z-x_{j,n})}{ r_\infty(z)}
          \right)^{1/2}
\end{equation}
where $U$ is a constant and the points $x_{j,n}$ and the
$\delta_{j,n} \in \{ \pm 1 \}$ are uniquely determined by the conditions that for
$k=1,...,l-1$
\begin{equation}\label{t9}
\begin{split}
    \sum\limits_{j=1}^{l-1} \delta_{j,n} \omega_{k}(x_{j,n}) =
    & - (2 n - l + 1) \omega_k(\infty) - \frac{2}{\pi} \int_E \log \left(
      \frac{w_0(\xi)}{\rho(\xi)} \right)
      \frac{\partial \omega_k(\xi)}{\partial n^{+}_\xi} d\xi \\
    & + 2 \sum\limits_{j=1}^m \omega_k(d_j) {\rm \qquad modulo \ } 2,
\end{split}
\end{equation}
recall that $\omega_k(x) := \omega(x; E_k, \tilde{\Omega})$ is the harmonic measure
of $E_k$ with respect to $\tilde{\Omega}$ at the point $x$ and $n^{+}_\xi$ is the
normal vector at $\xi$ pointing in the upper half plane. It turns out that $x_{j,n}
\in [a_{2j},a_{2j+1}]$ for $j = 1,...,l-1$ and thus
\begin{equation}\label{t10}
   \hat{g}_{(n)}(x;w) := \hat{g}_{(n)}(x) := \prod_{j=1}^{l-1} (x - x_{j,n})
\end{equation}
has exactly one zero in each gap $[a_{2j},a_{2j+1}],$ $j = 1,...,l-1.$ The
asymptotic representation \eqref{t8} is due to Widom
\cite[Theorem 6.2, p. 168 and pp. 175-176]{Wid} if there appear no point measures,
i.e., if $m = 0.$ The case of point measures has been first studied in \cite{Rah1}
for two intervals. Asymptotics for more general sets (so-called homogeneous sets)
and measures, including the above ones, have been given by P. Yuditskii and the
author in \cite{PehYud, PehYud2}.

The so-called Weyl solutions or functions of the second kind
\begin{equation}\label{t11}
   Q_n(y) = \int_E \frac{p_n(x)}{y - x} d\mu(x), {\rm \ respectively, \ } {\mathcal Q}_n(y) =
   \int_E \frac{P_n(x)}{y-x} d\mu(x)
\end{equation}
build another basis system of solutions of the recurrence relation \eqref{62}.
In \cite[Section 3]{PehYud} also asymptotics for Weyl solutions have been derived,
more precisely, using the terminology of this paper, the following uniform asymptotic
equivalence on compact subsets of $\Omega$ has been shown
\begin{equation}\label{t12}
    {\mathcal Q}_n(z) \sim \frac{\prod\limits_{j=1}^{m}\phi(z;d_j)}{U {\mathcal W}_0(z) \phi^{n+1}(z)}
    \left( \frac{\hat{g}_{(n)}(z)}{r_\infty(z) \prod_{j=1}^{l-1} \phi(z;c_j)
    \prod_{j=1}^{l-1} \phi(z;x_{j,n})^{\delta_{j,n}}}\right)^{1/2}
\end{equation}
where ${\mathcal W}_0 {\mathcal Q}_n$ is denoted by $h_n$ respectively $h$ in \cite{PehYud}.

Finally we mention that other questions about finite gap Jacobi matrices are under
investigations by J.S. Christiansen, B. Simon and M. Zinchenko \cite{ChrSimZin}
and that a huge part of the forthcoming book by B. Simon \cite{Sim1}, see also
\cite{Sim2}, is devoted to this topic. The approach is based on automorphic
functions similarly as in \cite{PehYud, PehYud2, SodYud} and is different
from that one used here.

\section{Asymptotics for principal and secondary diagonal Green's function}

The diagonal matrix elements of the resolvent $({\mathcal J} - z)^{-1}$
are the so-called Green's functions
$G(z,n,m) = \langle \delta_n,({\mathcal J} - z)^{-1} \delta_m \rangle,$
where as usual $\mathcal J$ denotes the Jacobi matrix associated with
$(\alpha_j)$ and $(\lambda_j).$ More precisely
using the orthogonality of $P_n,$ respectively, $P_{n + 1}$ one obtains
\begin{equation}\label{G1}
     P_n(z) {\mathcal Q}_n (z) = \int \frac{P^2_n(x)}{z - x} d\mu(x) =: G(z, n, n)
\end{equation}
\begin{equation}\label{G2}
     P_n(z) {\mathcal Q}_{n+1} (z) = \int \frac{P_{n+1}(x)P_n(x)}{z - x} d\mu(x) =: G(z, n+1, n)
\end{equation}
and
\begin{equation}\label{G3}
     P_{n+1}(z) {\mathcal Q}_n (z) =  \int \frac{P_{n+1}(x)P_n(x)}{z - x} d\mu(x) +
     \frac{1}{\sqrt{\lambda_{n + 2}}} =: G(z, n, n+1).\\
\end{equation}
Recall that in the case under consideration the $\lambda _n$'s are bounded away from zero. Using
Schwarz's inequality it follows that the three sequences
$(P_n {\mathcal Q}_n), (P_{n} {\mathcal Q}_{n+1})$
and $(P_{n+1} {\mathcal Q}_n)$ are bounded and analytic on compact subsets of
$\mathbb R \backslash {\rm supp}(\mu),$ hence they are so-called normal families.
Combining \eqref{t8} and \eqref{t12} we obtain with the help of \eqref{t2} the following asymptotics
for the diagonal Green's function.

\begin{cor}\label{cor1}
Let $\mu$ be given by \eqref{7_2} with $w \in {\rm Sz}(E).$ Then uniformly on compact
subsets of $\Omega$ there holds
\begin{equation}\label{t13}
   (P_n {\mathcal Q}_n)(z) = \frac{\hat{g}_{(n)}(z)}{\sqrt{H(z)}} + o(1).
\end{equation}
\end{cor}
For weight functions of the form $v(x) \rho(x) dx$ plus a possible finite number of
mass points outside $[a_1,a_{2l}],$ where $v$ is positive and analytic on $E,$ the
limit relation \eqref{t13} has been derived recently in \cite{Sue1} independently from
the above asymptotics \eqref{t8} and \eqref{t12}.

To derive the asymptotic representations of $G(z, n + 1, n)$ and $G(z, n, n+1)$
we will make use of Abel's Theorem and the solvability and uniqueness of
the real Jacobi inversion problem.
For this reason we have to write Widom's condition
\eqref{t9} in terms of Abelian integrals, which we will do
similarly as in \cite{Apt,Sue1}, see also \cite{AchTom}.

Let $\mathfrak R$ denote the hyperelliptic Riemann surface of genus $l-1$ defined
by $y^2 = H$ with branch cuts $[a_1,a_2], [a_2,a_3],\ldots,[a_{2l-1},a_{2l}].$
Points on $\mathfrak R$ are written in the form $\mathfrak z$ and $z$ denotes the
projection of $\mathfrak z$ on $\mathbb C$ also written ${\rm pr}({\mathfrak z}) = z.$
Furthermore we also use the notation ${\rm pr}({\mathfrak x}) = x$ and ${\rm pr}({\mathfrak y})
= y.$ $\mathfrak z$ and ${\mathfrak z}^{*}$ denote the points which lie above each
other on $\mathfrak R,$ i.e. ${\rm pr}({\mathfrak z}) = {\rm pr}({\mathfrak z}^*).$
The two sheets of $\mathfrak R$ are denoted by $\mathfrak R^+$ and $\mathfrak R^-.$
To indicate that ${\mathfrak z}$ lies on the first respectively second sheet we write
${\mathfrak z}^+$ and ${\mathfrak z}^-$. On ${\mathfrak R}^+$ the branch of $\sqrt{H}$
is chosen for which $\sqrt{H({\mathfrak x}^{+})} > 0$ for ${\mathfrak x}^+ > a_{2l}$.

Furthermore (see e.g. \cite{Osg,Spr}) let the cycles  $\{\alpha _j,\beta_j\}_{j=1}^{l-1}$
be the usual canonical homology basis  on $\mathfrak R$, i.e., the curve $\alpha _j$ lies
on the upper sheet $\mathfrak R^+$ of $\mathfrak R$ and encircles there
clockwise the interval $E_j$ and the curve $\beta _j$ originates at $a_{2j}$
arrives at $a_{2l-1}$ along the upper sheet and turns back to $a_{2j}$ along
the lower sheet. Let $\{ \varphi _1,\ldots,\varphi_{l-1}\}$, where $\varphi _j = \sum^{l-1}_{s=1}
e_{j,s} \frac{{\mathfrak z}^s}{\sqrt{H(\mathfrak z)}} d {\mathfrak z}, e_{j,s}\in\C$,
be a base of the normalized Abelian differential of the first kind, i.e.,
\begin{equation}
 	\int_{\alpha _j}\varphi _k = 2\pi i\delta _{jk} \quad \text{and}\quad
 	\int_{\beta_j}\varphi _k = B_{jk} \quad \text{for}\; j,k = 1,\ldots,l-1
 	\label{eq-38}
\end{equation}
where $\delta_{jk}$ denotes the Kronecker symbol here. The integrals in \eqref{eq-38}
are the so-called periods. Note that the $e_{j,n}$'s
are real since $\sqrt{H}$ is purely imaginary on $E_j$ and since
$\sqrt{H}$ is real on $\R \setminus E$ the symmetric matrix of periods $(B_{j,k})$ is real also.

Abelian differentials of third kind are differentials with simple poles at given points
$\mathfrak x$ and $\mathfrak y$ on $\mathfrak R$ with residues $+1$ and $-1,$ respectively,
and normalized such that integrals along the $\alpha_{\kappa}-$cycles vanish for
$\kappa = 1,...,l-1.$

Representing the differential $d \log \phi(z,c),$ $c \in \mathbb R  \backslash E,$
as linear combinations of normalized Abelian differentials of first and third kind (recall that
$\phi(\cdot, {\mathfrak c}^{+})$, where $c = {\mathfrak c}^{+},$ has a pole at ${\mathfrak c}^{+}$
and a zero at ${\mathfrak c}^{-}$ since the analytic extension of the complex Green's function
to the second sheet is given by $\phi({\mathfrak z}^{-}, {\mathfrak c}^{+}) =
1/\phi({\mathfrak z}^{+}, {\mathfrak c}^{+})$) one obtains by integrating along the $\alpha_j$
and $\beta_j$ cycles that
\begin{equation}\label{ep6t1}
    \int_{{\mathfrak c}^{-}}^{{\mathfrak c}^{+}} \varphi _j = \sum_{\kappa =1}^{l-1}
    \omega _\kappa(c)B_{j\kappa}.
\end{equation}
Similarly (see \cite{Pehmin} for details), since the analytic
extension of the harmonic measure
\begin{equation}\label{ep6x}
    w_k = \omega_k + i \omega^{*}_k {\rm \ \ \ on \ \ \ } {\mathfrak R}^{+} {\rm \ \ \ and \ \ \ }
    w_k = - \omega_k + i \omega^{*}_k {\rm \ \ \ on \ \ \ } {\mathfrak R}^{-}
\end{equation}
is just another basis of Abelian differentials of first kind, $\varphi_j$ can be represented
as linear combination of the $w_k$'s. Integrating along the $\beta_\kappa$ cycle,
$\kappa \in \{ 1,...,l-1 \},$ and recalling the fact that the integral along a $\beta_\kappa$-cycle
is the difference of values along the $\alpha_\kappa$
cycle, which is $E_\kappa,$ we obtain by \eqref{ep6x} that
$$ \int \varphi_j  = - \frac{1}{2} \sum_{\kappa = 1}^{l - 1} w_\kappa B_{j \kappa}.$$
Hence, using the fact that $\omega_k(z) = 1$ on $E_\kappa,$ $\kappa \in \{ 1,...,l \},$
and thus by the Cauchy-Riemann equations
$d w_\kappa(z) = i \frac{\partial \omega_\kappa}{\partial n} d s,$
we get
\begin{equation}\label{ep6t2}
   \frac{2}{\pi i} \varphi^{+}_{j} = \frac{1}{\pi} \sum_{\kappa = 1}^{l-1}
   \left( \frac{\partial \omega_\kappa}{\partial n^{+}} ds \right) B_{j \kappa}
\end{equation}

\begin{notation}
Let $[a_{2j},a_{2j+1}]^{\pm}$ denote the two copies of $[a_{2j},a_{2j+1}],$
$j = 1, ..., l-1$ in ${\mathfrak R}^{\pm}.$ Note that $[a_{2j},a_{2j+1}]^+ \cup
[a_{2j},a_{2j+1}]^-$ is a closed loop on ${\mathfrak R}$ and thus
$\sf{X}_{j=1}^{l-1}$ $( [a_{2j}, a_{2j+1}]^+ \cup$ $[a_{2j},a_{2j+1}]^-)$
is topologically a $l-1$ dimensional torus.
\end{notation}

By \eqref{ep6t1} and \eqref{ep6t2} Widom's condition \eqref{t9} becomes,
\begin{equation}\label{nw}
\begin{split}
    \frac{1}{2} \sum\limits_{j = 1}^{l - 1} \int_{{\mathfrak x}_{j,n}^{*}}^
    {{\mathfrak x}_{j,n}} \varphi_k =
    & - (n - \frac{l-1}{2}) \int_{\infty^{-}}^{\infty^{+}} \varphi_k
      - \frac{i}{\pi} \int_E \varphi_k^{+} \log w \\
    & + \sum\limits_{j=1}^{m} \int_{{\mathfrak d}_j^{-}}^{{\mathfrak d}_j^{+}}
      \varphi_k {\rm \ \ \ \ \ mod \ periods}
\end{split}
\end{equation}
where ${\rm pr} ({\mathfrak x}_{j,n}) = x_{j,n}$ and ${\mathfrak x}_{j,n}
\in {\mathfrak R}^{\delta_{j,n}} := {\mathfrak R}^{\pm}$ for $\delta_{j,n} = \pm 1;$
recall that ${\mathfrak x}_{j,n}$ and ${\mathfrak x}_{j,n}^*$ lie above each other.
Furthermore ${\rm pr} ({\mathfrak d}_{j}^{\pm}) = d_{j}.$

Next we note that \eqref{nw} can be considered as so-called Jacobi inversion problem,
that is, for given $(\eta_1,...,\eta_{l-1}) \in {\rm Jac \ } {\mathfrak R},$ where
${\rm Jac \ } {\mathfrak R}$ denotes the Jacobi variety of $\mathfrak R,$ (that is the
quotient space ${\mathbb C}^{l-1} \backslash (2 \pi i \overrightarrow{n} + B
\overrightarrow{m})$, $B = (B_{j k})$ the matrix of periods, $\overrightarrow{n},
\overrightarrow{m} \in {\mathbb Z}^{l-1}$) find ${\mathfrak z}_1,...,
{\mathfrak z}_{l-1} \in \mathfrak R$ such that $$ \sum_{j = 1}^{l - 1}
\int_{{\mathfrak e}_j}^{{\mathfrak z}_j} \varphi_k  = \eta_k
{\rm \ \ \ mod \ periods }, $$ where ${\mathfrak e}_1,...,{\mathfrak e}_{l-1}$ are
given points on $\mathfrak R.$ In this paper the $\eta_j$'s are real always, that is,
we deal with the real Jacobi inversion problem. Denoting by ${\rm Jac \ } {\mathfrak R}
/ {\mathbb R} := {\mathbb R}^{l-1}/ B \overrightarrow{m}$ the Jacobi variety restricted
to reals, the following important uniqueness property of the real Jacobi inversion problem
holds (see e.g. \cite{Kre-Lev-Nud,PehIMNR}): The restricted Abel map
\begin{equation}\label{ct}
\begin{split}
   & {\mathcal A} : {\sf{X}}_{j=1}^{l-1} ([a_{2j},a_{2j+1}]^{+} \cup [a_{2j},a_{2j+1}]^{-})
     \to {\rm Jac \ } {\mathfrak R}/{\mathbb R} \\
   & \qquad \qquad \qquad \qquad \qquad \ ({\mathfrak z}_1,...,{\mathfrak z}_{l-1}) \mapsto
     \frac{1}{2} \left(
     \sum_{j=1}^{l-1} \int_{{\mathfrak z}_j^{*}}^{{\mathfrak z}_j} \varphi_1,...,
     \sum_{j=1}^{l-1} \int_{{\mathfrak z}_j^{*}}^{{\mathfrak z}_j} \varphi_{l-1} \right)
\end{split}
\end{equation}
is a holomorphic bijection.

Now let us show that the solutions ${\mathfrak x}_{1,n},...,{\mathfrak x}_{l-1,n}$ of \eqref{nw}
can be described uniquely with the help of two polynomials which
is crucial in what follows.

\begin{prop}\label{prop1}
a) Let $g_{(n)},$ $n \in \mathbb N,$ be given by \eqref{t10}.
There exists a monic polynomial $\hat{f}_{(n+1)}$ of degree $l$ such that
\begin{equation}\label{P1}
    \hat{f}_{(n+1)}^2(x) - H(x) = L_n \hat{g}_{(n+1)}(x) \hat{g}_{(n)}(x)
\end{equation}
with
\begin{equation}\label{P2}
    \hat{f}_{(n+1)}(x_{j,n}) = - \delta_{j,n} \sqrt{H(x_{j,n})} {\rm \ and \ }
    \hat{f}_{(n+1)}(x_{j,n+1}) = - \delta_{j,n+1} \sqrt{H(x_{j,n+1})}
\end{equation}
where $x_{j,n}, x_{j,n+1}, \delta_{j,n}, \delta_{j,n+1}$ are given by \eqref{t9} and
$L_n \in \mathbb R$ is given below. Moreover $\hat{f}_{(n+1)}$ can be represented in
the form
\begin{equation}\label{P3}
   \hat{f}_{(n+1)}(x) = \left( x + \sum\limits_{j = 1}^{l - 1} x_{j,n} - c_1
                        \right) \hat{g}_{(n)}(x) -
   \sum\limits_{j = 1}^{l - 1} \frac{\delta_{j,n} \sqrt{H(x_{j,n})} \hat{g}_{(n)}(x)}
   {\hat{g}^{'}_{(n)}(x_{j,n})(x - x_{j,n})},
\end{equation}
where $c_1$ is given by $H(x) = x^{2 l} - 2 c_1 x^{2 l - 1} + ... .$

Furthermore
\begin{equation}\label{P4}
    \frac{(\hat{f}_{(n+1)} \pm  \sqrt{H})^2(z)}{L_n \hat{g}_{(n)}(z) \hat{g}_{(n+1)}(z)} =
    \left( \phi^2(z;\infty) \prod\limits_{j=1}^{l-1}
    \frac{\phi(z;x_{j,n+1})^{\delta_{j,n+1}}}{\phi(z;x_{j,n})^{\delta_{j,n}}}
    \right)^{\pm 1}
\end{equation}
and
\begin{equation}\label{P5}
\begin{split}
    L_n & = 4 ({\rm cap} (E))^2 \prod\limits_{j=1}^{l-1}
            \frac{\phi(x_{j,n};\infty)^{\delta_{j,n}}}
            {\phi(x_{j,n+1};\infty)^{\delta_{j,n+1}}} = 4 \lambda_{n+2} (1+o(1))
\end{split}
\end{equation}

b) To each solution $({\mathfrak x}_{1,n},...,{\mathfrak x}_{l-1,n})$ of \eqref{nw}
there corresponds a unique pair of polynomials $(\hat{g}_{(n)},\hat{f}_{(1+n)})$
given by \eqref{t10} and by \eqref{P3} with $\sqrt{H({\mathfrak x}_{j,n})} =
\delta_{j,n} \sqrt{H(x_{j,n})}.$
\end{prop}

\begin{proof}
a) Let us write ${\mathfrak x}^{\delta_{j,n}}_{j,n} := {\mathfrak x}_n^{\pm}$ for
$\delta_{j,n} = \pm 1.$ Considering \eqref{nw} for $n$ and $n+1$ and substracting
both equations we obtain that
$$ \sum\limits_{j = 1}^{l - 1} \int_{{\mathfrak x_{j,n+1}^{\delta_{j,n+1}}}}
   ^{{\mathfrak x_{j,n+1}^{- \delta_{j,n+1}}}} \varphi_k
   = 2 \int_{\infty^{-}}^{\infty^{+}} \varphi_k +
   \sum\limits_{j = 1}^{l - 1} \int_{{\mathfrak x_{j,n}^{\delta_{j,n}}}}^
   {{\mathfrak x_{j,n}^
   {- \delta_{j,n}}}} \varphi_k, \ k = 1,...,l-1.
$$
Thus it follows from Abel's Theorem that there is a rational function
${\mathcal R}_n$ on the Riemann surface with the following properties
\begin{equation}\label{x1}
\begin{split}
   & \infty^{{\pm}} {\rm \ is \ a \ double \ pole \ (zero) }\\
   & {\mathfrak x}_{j,n+1}^{\mp \delta_{j,n+1}}
     {\rm \ is \ a \ simple \ zero \ (pole) }\\
   & {\mathfrak x}_{j,n}^{\mp \delta_{j,n+1}}
     {\rm \ is \ a \ simple \ pole \ (zero)}.
\end{split}
\end{equation}
Thus ${\mathcal R}_n$ can be represented in the form
\begin{equation}\label{x2}
   {\mathcal R}_n = \frac{(f_{(n+1)} +  \sqrt{H})^2}{L_n \hat{g}_{(n)} \hat{g}_{(n+1)}},
\end{equation}
where $f_{(n+1)}$ is a polynomial of degree $l$ and $L_n$ is a constant, since
the function at the RHS in \eqref{x2} satisfies \eqref{x1} and has no other zeros
or poles on ${\mathfrak R}$ as can be checked easily. Since the points
${\mathfrak x}_{j,n+1}^{\pm \delta_{j,n+1}}$ and ${\mathfrak x}_{j,n}^{\pm \delta_{j,n}}$
are real the rational function $\overline{\mathcal R_n({\mathfrak z})}$ has the same
properties \eqref{x1} as $\mathcal R_n(\bar{{\mathfrak z}})$, hence $\mathcal
R_n(\bar{{\mathfrak z}}) = \overline{\mathcal R_n({\mathfrak z})}$ and therefore
$f_{(n+1)}$ has real coefficients and $L_n \in \mathbb R.$ Taking involution, denoted
by ${\mathfrak z}^*,$ we obtain by \eqref{x1}
\begin{equation}\label{x3}
    \frac{1}{{\mathcal R}_n(\mathfrak z)} = {\mathcal R}_n({\mathfrak z}^*) =
    \frac{(f_{(n+1)} - \sqrt{H})^2}{L_n \hat{g}_{(n)} \hat{g}_{(n+1)}},
\end{equation}
and thus, by multiplying \eqref{x3} and \eqref{x2}, relation \eqref{P1}-\eqref{P2}.
Concerning \eqref{P2} note that $\sqrt{H({\mathfrak x}^{\delta})} = \delta \sqrt{H(x)}.$

Next let us note
$$ - (\sqrt{H})^{-}(x) =  (\sqrt{H})^{+}(x) =
i (-1)^{l-k} \sqrt{|H(x)|} {\rm \ for \ } x \in E_k$$
and therefore by \eqref{x2} and \eqref{P1}
$$ |{\mathcal R}^{\pm}(x)| = 1 {\rm \ for \ } x \in E.$$

Thus the function
$$ F(z) := \frac{{\mathcal R}_n(z)}{\phi^2(z;\infty)}
   \prod\limits_{j=1}^{l-1} \frac{\phi(z;x_{j,n})^{\delta_{j,n}}}
   {\phi(z;x_{j,n+1})^{\delta_{j,n+1}}} $$
has neither zeros nor poles on $\Omega$ and satisfies $|F^{\pm}| = 1$ on $E.$
Hence $\log |f(z)|$ is a harmonic bounded function on $\Omega$ which has a
continuous extension to $E$ and thus $F = 1,$ which proves \eqref{P4}.

Considering \eqref{P4} at $z = \infty$ the first relation in $\eqref{P5}$ follows.
Next denote by $lc(P_n)$ the leading coefficient of $P_n.$ Obviously
$  \lambda_{n + 2} = \int p^2_{n+1} / \int p_n^2 = \left(  lc(P_n) /
lc(P_{n+1}) \right)^2.$ Using the fact that $lc(P_n) = \lim\limits_{z \to \infty} P_n(z) / z^n$ it follows
by \eqref{t8}, in conjunction with \eqref{C}, that the last equality in \eqref{P5} holds.

Relation \eqref{P3} follows by \eqref{P2}, note that the second term at the
RHS of \eqref{P3} is the unique Lagrange interpolation polynomial which takes
on at $x_{j,n}$ the value $- \delta_{j,n} \sqrt{H(x_{j,n})},$ and by
equating the first two leading coefficients in \eqref{P1}.

b) follows immediately by a).
\end{proof}

We note in passing that \eqref{P5}, which may be derived from \eqref{t8} also,
is the so-called trace formula for $\lambda_{n+2}$ and that the other trace formula
$- \alpha_{n + 1} = \sum_{j = 1}^{l-1} x_{j,n} - c_1 + o(1)$ follows immediately by
\eqref{G1}, \eqref{t13} and by equating coefficients of $x^{l-1}$ in \eqref{P3}. By
\eqref{nw} and \eqref{ct} the $x_{j,n}$'s may be expressed with the help of the Abel
map from which representation \eqref{F1} can be obtained.

For measures from $\mathcal{G}$, $\mathcal{G}$ defined in Corollary \ref{cor4.4},
which are so-called reflectionless measures (note that the set of Jacobi matrices
associated with the set of measures $\mathcal{G}$ is the so called isospectral torus),
Proposition \ref{prop1} is essentially known \cite{Mag, TurDuc, PehSIAM, Tes} apart from the
important relation \eqref{P4}. But the way of derivation is exactly the opposite
one. First one shows that certain polynomials, which are expressions in Pad\'{e}
approximants, satisfy \eqref{P1} - \eqref{P3} - to find the corresponding polynomials
in the general case seems to be hopeless, if they exist at all - and then one derives
the correspondence with the solutions of \eqref{nw}.

\begin{thm}\label{under}
Let $\mu$ be given by \eqref{7_2} with $w \in {\rm Sz}(E).$ Then uniformly on compact
subsets of $\Omega$ there holds
\begin{equation}\label{a1}
   2 \sqrt{\lambda_{n+2}} P_n {\mathcal Q}_{n+1} = \frac{\hat{f}_{(n+1)} -
   \sqrt{H}}{\sqrt{H}}+o(1)
\end{equation}
\begin{equation}\label{a2}
   2 \sqrt{\lambda_{n+2}} P_{n+1} {\mathcal Q}_{n} = \frac{\hat{f}_{(n+1)} +
   \sqrt{H}}{\sqrt{H}}+o(1)
\end{equation}
and on compact subsets of $\C \setminus [a_1,a_{2l}]$
\begin{equation}\label{a4}
  2\sqrt{\lambda_{n+2}} \frac{P_{n+1}}{P_{n}} = \frac{\hat{f}_{(n+1)} +
   \sqrt{H}}{\hat{g}_{(n)}} + o(1)
\end{equation}

\begin{equation}\label{a3}
   \int \frac{1}{z - x} d \mu^{(n+1)}(x) = \sqrt{\lambda_{n+2}} \frac{{\mathcal Q}_{n+1}(z)}{{\mathcal Q}_{n}(z)} =
   \frac{\hat{f}_{(n+1)}(z) - \sqrt{H}(z)}{2\hat{g}_{(n)}(z)} + o(1),
\end{equation}
where $\mu^{(m+1)}$ denotes the measure associated with the $m+1$ forwards shifted recurrence
coefficients $(\alpha_{n+m+1})_{n \in \mathbb N}$ and $(\lambda_{n+m+2})_{n \in \mathbb N}$.
\end{thm}

\begin{proof}

Concerning relation \eqref{a1}. Plugging in the asymptotic values for $P_n$ and
${\mathcal Q}_{n+1}$ from \eqref{t8} and \eqref{t12} and recalling the fact that
$\phi^{'}/\phi = r_\infty / \sqrt{H}$ the asymptotic relation follows immediately
by \eqref{P4} and \eqref{P5}.

Analogously \eqref{a2} is proved. \eqref{a3} and \eqref{a4} follow immediately by
\eqref{a1} and Corollary \ref{cor1}, respectively \eqref{a2} and Corollary \ref{cor1}.
\end{proof}

\begin{cor}\label{corR31}
In a neighborhood of $x = 0$ let
$$ \sqrt{H^{*}(x)} = \sum\limits_{j = 0}^{\infty} h_j x^j $$
where $H^{*}(x) = x^{2l} H(\frac{1}{x})$ is the reciprocal polynomial of $H$.
Then the following limit relations hold for $m \geq l$
\begin{equation}\label{R31x1}
    \lim\limits_n \left( \sum\limits_{j = 0}^{m} h_j \int t^{m-j} P^2_n \right) = 0
\end{equation}
and for $m \geq l+1$
\begin{equation}\label{R31x2}
    \lim\limits_n \left( h_m + 2 \sum\limits_{j = 0}^{m - 1} h_j \sqrt{\lambda_{2 + n}}
    \int t^{m - 1 - j} P_{1 + n} P_n \right) = 0
\end{equation}
\end{cor}

\begin{proof}
By \eqref{t13} and \eqref{G1} it follows that in a neighborhood of $x = 0$
\begin{equation}\label{x5}
    \sqrt{H^{*}(x)} \left( \sum\limits_{j = 0}^{\infty} (\int t^j P^2_n) x^j \right) =
    \hat{g}^*_{(n)}(x) + o(1)
\end{equation}
Equating coefficients for $m \geq l$ relation \eqref{R31x1} follows.

Concerning the second relation we get by \eqref{a1} that
\begin{equation}\label{x6}
    \sqrt{H^{*}(x)} \left( 1 + 2 \sum\limits_{j = 0}^{\infty} ( \sqrt{\lambda_{2 + n}}
    \int t^j P_{1 + n} P_n) x^{j+1}  \right)  = \hat{f}^*_{(1 + n)}(x) + o(1)
\end{equation}
which proves \eqref{R31x2}.
\end{proof}

Most likely the limit relations \eqref{R31x1} and \eqref{R31x2} hold for
measures $\sigma$ whose essential support is $E$. For the subclass of measures
$\mu \in {\mathcal G}$
it can be shown by a different approach, that \eqref{R31x1} and \eqref{R31x2}
hold for $m=l$ and $m=l+1$, respectively, without limit even, see also \cite{Mag,TurDuc}
where the relations are derived in terms of recurrence coefficients for $l=2,3$.

\section{Accumulation points of moments of the Green's functions}

By series expansion of $1 / \sqrt{H(z)}$ at $z = \infty$
\begin{equation}\label{S10}
   \sum_{j=0}^{\infty} c_j z^{-(l + j)} = \frac{1}{\sqrt{H(z)}} =
   \frac{1}{\pi} \int_E \frac{1}{z - t} \frac{dt}{h(t)}
\end{equation}
for $z \in {\mathbb C} \backslash E,$ where the last equality follows by the
Sochozki-Plemelj formula. In particular
\begin{equation}\label{lemmaEx3}
   \int_E x^j \frac{dx}{h(x)} = 0 {\rm \ for \ } j = 0,...,l-2 {\rm \ and \ }
   c_0 = \int_E x^{l-1} \frac{dx}{h(x)} = 1.
\end{equation}
Hence the following lemma holds.

\begin{lemma}\label{lemmaE}
a) Let $A_1,...,A_{l-1} \in \mathbb R$ be given. There exists an unique polynomial
$P(x) = \sum\limits_{\nu = 0}^{l - 1} p_\nu x^\nu$ with $p_{l-1} = 1$
such that
\begin{equation}\label{lemmaEx1}
    \int_E x^j P(x) \frac{dx}{h(x)} = A_j {\rm \ \ \ for \ \ \ } j = 1,...,l-1.\\
\end{equation}
b) Let $B_1,...,B_{l-1} \in \mathbb R$ be given. There exists an unique polynomial $Q(x)$
of degree $l$ with two fixed leading coefficients such that
\begin{equation}\label{lemmaEx2}
    \int_E x^j Q(x) \frac{dx}{h(x)} = B_j {\rm \ \ \ for \ \ \ } j = 1,...,l-1.\\
\end{equation}
\end{lemma}

Now we are ready to state our first main result.

\begin{thm}\label{thm4.1}
Let $\mu$ be given by \eqref{7_2} and let $w \in {\rm Sz}(E).$ a) The subsequence
of solutions $({\mathfrak x}_{1,n_\nu},...,{\mathfrak x}_{l-1,n_\nu} )_{\nu \in \mathbb N}$
of \eqref{nw} converges if and only if $( \int x P^2_{n_\nu} d\mu, ...,
\int x^{l-1} P^2_{n_\nu} d\mu,$ $\sqrt{\lambda_{2 + n_\nu}} \int x P_{1 + n_\nu} P_{n_\nu}d\mu,$
$... ,\sqrt{\lambda_{2 + n_\nu}}\int x^{l-1}$ $P_{1 + n_\nu}$
$P_{n_\nu} d\mu )_{\nu \in \mathbb N}$ converges.

Furthermore, the map $\tau,$ given by
\begin{equation}\label{thmMt0}
\begin{split}
  \left( {\mathfrak y}_1, ..., {\mathfrak y}_{l-1}  \right) \mapsto
   & \left( \int_E x G(x) \frac{dx}{h(x)},..., \int_E x^{l-1} G(x) \frac{dx}{h(x)}, \right.\\
  & \left. \int_E \frac{x F(x)}{2} \frac{dx}{h(x)},..., \int_E \frac{x^{l-1} F(x)}{2}
           \frac{dx}{h(x)} \right)
\end{split}
\end{equation}
where,
\begin{equation}\label{thmMt1}
    G(x) := \prod\limits_{j = 1}^{l - 1} (x - y_j), \ F(x) := \left( x + \sum y_j  -
    c_1 \right) G(x) - \sum\limits_{j = 1}^{l - 1} \frac{\delta_j \sqrt{H(y_j)}
    G(x)}{G^{'}(y_j) (x - y_j)},
\end{equation}
$c_1$ as in \eqref{P3} and $\delta_j \sqrt{H(y_j)} = \sqrt{H({\mathfrak y}_j)}$,
is a homeomorphism between the set of accumulation points of the
sequence of solutions $({\mathfrak x}_{1,n},...,$ ${\mathfrak x}_{l-1,n})_{n \in \mathbb N}$
of \eqref{nw} and of the sequence $( \int x P^2_{n} d\mu, ...,$ $ \int x^{l-1} P^2_{n} d\mu,$
$\sqrt{\lambda_{2 + n}}$ $\int x $ $P_{1 + n} P_{n}d\mu,$ $... ,\sqrt{\lambda_{2 + n}}\int x^{l-1}
P_{1 + n} P_{n}d\mu )_{n \in \mathbb N}$

b) If $1, \omega_1(\infty),..., \omega_{l-1}(\infty)$ are linearly independent
over $\mathbb Q,$ then $\tau$ is a homeomorphism from $\sf{X}_{j=1}^{l-1}$
$( [a_{2j}, a_{2j+1}]^+ \cup$ $[a_{2j},a_{2j+1}]^-)$ into the set of accumulation
points of $( \int x P^2_{n} d\mu, ...,$ $ \int x^{l-1} P^2_{n} d\mu,$
$\sqrt{\lambda_{2 + n}}$ $ \int x P_{1 + n} P_{n}d\mu,$
$... ,$ $\sqrt{\lambda_{2 + n}}\int x^{l-1} P_{1 + n} P_{n}d\mu )_{n \in \mathbb N}.$
\end{thm}

\begin{proof}
At the beginning let us observe that the map $\tau$ is a continuous one to one map
from $\sf{X}_{j=1}^{l-1}$ $([a_{2j},a_{2j+1}]^{+} \cup [a_{2j},a_{2j+1}]^{-})$
on its range. Indeed, recall first the obvious fact that there is a unique
correspondence between a point ${\mathfrak y} \in \sf{X}_{j=1}^{l-1}$
$([a_{2j},a_{2j+1}]^{+} \cup [a_{2j},a_{2j+1}]^{-})$ and
$(y,\delta \sqrt{H(y)}) = ({\rm pr}({\mathfrak y}), \sqrt{H({\mathfrak y})} ),$
where $\delta = \pm 1$ if ${\mathfrak y} \in {\mathfrak R}^{\pm}.$
Since the polynomial $F(x) - \left( x + (\sum y_j - c_1) \right) \cdot$ $G(x)$ is
the unique Lagrange interpolation polynomial of degree $l - 2$ which takes on at
the $y_j$'s the values $- \delta_j \sqrt{H(y_j)},$ it follows that there is a
unique correspondence between the points
$\left( {\mathfrak y}_1, ... , {\mathfrak y}_{l-1} \right)$ and the polynomials
$(G,F),$ defined in \eqref{thmMt1}. By Lemma \ref{lemmaE} the observation is proved.

a) Necessity of the first statement. Let $({\mathfrak x}_{1,n_\nu},...,
{\mathfrak x}_{l-1,n_\nu})_{\nu \in \mathbb N}$ be a sequence of solutions of
\eqref{nw} with limit $({\mathfrak y}_1,...,{\mathfrak y}_{l-1}),$ that is,
$j = 1,...,l-1,$
$$ x_{j, n_\nu} \underset {\nu \to \infty} \longrightarrow y_j
{\rm \ and \ }
\delta_{j,n_\nu} \sqrt{H(x_{j,n_\nu})} \underset {\nu \to \infty}
\longrightarrow \delta_{j} \sqrt{H(y_{j})}.$$
By Proposition \ref{prop1} it follows that there are unique polynomials
$\hat{g}_{(n_\nu)}$ and $\hat{f}_{(1 +  n_\nu)}$ such that uniformly on
$\Omega$ \begin{equation}\label{M3x1}
   \hat{g}_{(n_\nu)}(x) \underset {\nu \to \infty} \longrightarrow G(x) {\rm \ and \ }
   \hat{f}_{(1 + n_\nu)}(x) \underset {\nu \to \infty} \longrightarrow F(x),
\end{equation}
where for the second relation we took \eqref{P3} into account and that $G$ and $F$ are
given by \eqref{thmMt1}.

Next it follows by Corollary \ref{cor1} and \eqref{M3x1} that uniformly on compact
subsets of $\Omega$
\begin{equation}\label{M3x2}
\begin{split}
    \lim\limits_\nu P_{n_\nu} {\mathcal Q}_{n_\nu}(z)
    & = \lim\limits_\nu \int \frac{P^2_{n_\nu}(x)}
      {z - x} d\mu(x) = \frac{G(z)}{\sqrt{H(z)}}\\
    & = \frac{1}{\pi} \int_E \frac{G(x)}{z - x} \frac{dx}{h(x)}
\end{split}
\end{equation}
where the last equality follows by the Sochozki-Plemelj's formula using the fact
that $G \in {\mathbb P}_{l-1}.$ Analogously we obtain by Theorem \ref{under} that
\begin{equation}\label{M3x3}
\begin{split}
    \lim\limits_\nu \sqrt{\lambda_{2 + n_\nu}} P_{n_\nu} {\mathcal Q}_{1 + n_\nu}(z)
    & = \lim\limits_\nu \sqrt{\lambda_{2+n_\nu}} \int \frac{P_{n_\nu}(x) P_{1 + n_\nu}(x)}
        {z - x} d\mu(x) \\
    & = \frac{F(z) - \sqrt{H(z)}}{2 \sqrt{H(z)}} = \frac{1}{2 \pi}
        \int_E \frac{F(x)}{z - x} \frac{dx}{h(x)}
\end{split}
\end{equation}
where the last equality follows by Sochozki-Plemelj again and the fact that $-1 +
(F/\sqrt{H})(z) = O(\frac{1}{z})$ as $z \to \infty,$ since $f_{(1 + n_\nu)}$ has this
property by \eqref{P1}. Moreover the first $l-1$ coefficients of the series expansions
in \eqref{M3x2} and \eqref{M3x3} converge which proves the necessity part of the first
statement of a).

Furthermore, by \eqref{M3x2} and \eqref{M3x3} we have shown also, that
$$\lim\limits_\nu \left( \int x P^2_{n_\nu}, ..., \int x^{l-1} P^2_{n_\nu},
\int x P_{1 + n_\nu} P_{n_\nu}, ... ,
\int x^{l-1} P_{1 + n_\nu} P_{n_\nu} \right)$$
is in the range of the restricted map $\tau_/$ whose domain is the set of
accumulation points of the solutions of \eqref{nw}.

Sufficiency of the first statement of a). Suppose that
$\lim\limits_\nu \left( \int x P^2_{n_\nu}, ...,\right.$ $\int x^{l-1} $ $P^2_{n_\nu},$
$\sqrt{\lambda_{2 + n_\nu}}$ $\int x P_{1 + n_\nu} P_{n_\nu}, ... ,$
$\sqrt{\lambda_{2 + n_\nu}}$ $\int x^{l-1} $ $\left. P_{1 + n_\nu} P_{n_\nu} \right)$
exists. By \eqref{R31x1} and \eqref{R31x2} and induction arguments it
follows that $\lim_\nu \int x^j P^2_{n_\nu}$ and
$\lim_\nu \sqrt{\lambda_{2 + n_\nu}} \int x^j P_{1 + n_\nu} P_{n_\nu}$
exist for every $j \in \mathbb N$ and thus $\lim_\nu \int \frac{P^2_{n_\nu}}{z - x}$
and $\lim_\nu \sqrt{\lambda_{2+n_\nu}} \int \frac{P_{1+n_\nu}P_{n_\nu}}{z - x}$ is
uniformly convergent at a neighborhood of $z = \infty$ and thus on compact subsets of
$\mathbb C \backslash {\rm supp}(\mu).$  Corollary \ref{cor1} and Theorem \ref{under}
imply that the relations \eqref{M3x1}-\eqref{M3x3} hold, where $G, F$ are by Lemma
\ref{lemmaE} uniquely determined by $\lim\limits_\nu \left( \int x P^2_{n_\nu}, ...,
\right.$ $\int x^{l-1}$ $P^2_{n_\nu},$ $\sqrt{\lambda_{2 + n_\nu}} \int x P_{1 + n_\nu}
P_{n_\nu}, ... ,$ $\sqrt{\lambda_{2 + n_\nu}}$ $\left. \int x^{l-1} P_{1 + n_\nu}
P_{n_\nu} \right).$

Now let us take a look at the sequence of solutions of \eqref{nw}
$({\mathfrak x}_{1,n_\nu},...,$ ${\mathfrak x}_{l-1,n_\nu})_{\nu \in \mathbb N}$
which can be considered as the sequence
$\left( (x_{1,n_\nu}, \delta_{1,n_\nu} \sqrt{H(x_{1,n_\nu})}),\right.$
$\left. ...,(x_{l-1,n_\nu}, \delta_{l-1,n_\nu} \sqrt{H(x_{l-1,n_\nu})}) \right)_{\nu \in \mathbb N},$
where $\delta_{j,n_\nu} = \pm 1$ if ${\mathfrak x}_{j,n_\nu} \in {\mathfrak R}^{\pm}.$
Then it follows by Proposition \ref{prop1}
if $({\mathfrak x}_{1,n_\nu},...,{\mathfrak x}_{l-1,n_\nu})$
has two different limit points then the uniquely associated sequence of polynomials
$(\hat{g}_{(n_\nu)},\hat{f}_{(1+n_\nu)})$ has two different limit points. But this
contradicts \eqref{M3x1}. Hence there exists
$\lim_\nu ({\mathfrak x}_{1,n_\nu},...,{\mathfrak x}_{l-1,n_\nu}) =
({\mathfrak y}_{1},...,{\mathfrak y}_{l-1})$ and the first statement of a) is proved.

Furthermore, since \eqref{M3x2} - \eqref{M3x3} hold,
$({\mathfrak y}_{1},...,{\mathfrak y}_{l-1})$ is mapped by the continuous map $\tau_/$
to $\lim\limits_\nu \left( \int x P^2_{n_\nu}, ...,\right.$ $ \int x^{l-1} P^2_{n_\nu},
\sqrt{\lambda_{2 + n_\nu}} \int x P_{1 + n_\nu} P_{n_\nu}, ... ,$
$\sqrt{\lambda_{2 + n_\nu}}$ $\int x^{l-1} P_{1 + n_\nu} P_{n_\nu} \left. \right),$ that is,
$\tau_/$ is an onto map between the two sets of accumulation points and the second statement
of a) is proved.

b) We claim that the set of accumulation points of the sequence of solutions of \eqref{nw}
is the set $\sf{X}_{j=1}^{l-1}$ $([a_{2j},a_{2j+1}]^{+} \cup [a_{2j},a_{2j+1}]^{-}).$
Obviously it suffices to show that for given $\mbox{\boldmath$\mathfrak y$}=
(\mathfrak y_1,...,\mathfrak y_{l-1}) \in
\sf{X}_{j=1}^{l-1}$ $([a_{2j},a_{2j+1}]^{+} \cup [a_{2j},a_{2j+1}]^{-})$
there exists a sequence $\mbox{\boldmath$\mathfrak x$}_{n_\nu} :=
({\mathfrak x}_{1,n_\nu},...,{\mathfrak x}_{l-1,n_\nu})_{\nu \in \mathbb N}$
of \eqref{nw} such that $\lim_\nu \mbox{\boldmath$\mathfrak x$}_{n_\nu} =
\mbox{\boldmath$\mathfrak y$}.$ Since
$1,\omega_1(\infty),...,\omega_{l-1}(\infty)$ is linearly independent it follows
by Kronecker's Theorem \cite{Hla,Kok} that any sequence
$( - n \mbox{\boldmath$\mathfrak \omega$}(\infty)
+ \mbox{\boldmath$c$})_{n \in \mathbb N},$ where $\mbox{\boldmath$c$} =
(c_1,...,c_{l-1})$ is an arbitrary constant, is, modulo $1,$ dense in $[0,1]^{l-1}$.
Thus by the equivalence of \eqref{t9} and \eqref{nw}
$( - n \int_{- \infty}^{\infty} \mbox{\boldmath$\varphi$} +
\tilde{\mbox{\boldmath$c$}})_{n \in \mathbb N}$ is modulo periods $B_{j k},$
dense in ${\rm Jac \ } {\mathfrak R}/{\mathbb R} $ where the constant
$\tilde{\mbox{\boldmath$c$}} = (\tilde{c}_1,...,\tilde{c}_{l-1})$ is given
by that part of the RHS of \eqref{nw} which does not depend on $n$ and where
we have put $\mbox{\boldmath$\varphi$} = (\varphi_1,...,\varphi_{l-1}),$.
Hence there exists a subsequence $(n_\nu)$
of natural numbers such that
\begin{equation}\label{need}
     \left( - n_\nu \int_{- \infty}^{\infty} \mbox{\boldmath$\varphi$}
     +  \tilde{\mbox{\boldmath$c$}}\right)_{\nu \in \mathbb N} \to
     {\mathcal A}(\mbox{\boldmath$\mathfrak y$})\quad \quad \text{modulo periods}\; B_{jk},
     \end{equation}
where ${\mathcal A}$ is the Abel map from \eqref{ct}. On the other hand
to the sequence $(P_{n_\nu})_{\nu \in \mathbb N}$ of orthonormal polynomials
there exist points $\mbox{\boldmath$\mathfrak x$}_{n_\nu} =
({\mathfrak x}_{1,n_\nu},...,{\mathfrak x}_{l-1,n_\nu})$ such that \eqref{nw}
holds, that is, that
$$ {\mathcal A}( \mbox{\boldmath$\mathfrak x$}_{n_\nu} )
= - n_\nu \int_{-\infty}^{+ \infty}  \mbox{\boldmath$\varphi$}  +
\tilde{\mbox{\boldmath$c$}} \quad \quad \text{modulo periods}\; B_{jk},$$
By \eqref{need} and the bijectivity of ${\mathcal A}$ the assertion follows.
\end{proof}

\begin{remark}
If one wants to get rid of the Riemann-surface the homeomorphism from Theorem \ref{thm4.1}
a) may be written also as the map from $A := \{ ((y_1, \delta_1),...,$
$(y_{l-1},\delta_{l-1})): y_j \in [a_{2j},a_{2j+1}],$ $\delta_j \in \{ -1,1 \} {\rm \ if \ }
y_j \in (a_{2j},a_{2j+1}) {\rm \ and \ } \delta_j = 0 {\rm \ if \ } y_j \in $
$\{ a_{2j},a_{2j+1} \} {\rm \ for \ } j = 1,...,l-1 \}$ into the set of accumulation points
of the sequence $( \int x P^2_n, ...,$ $\int x^{l-1} P^2_n,$ $\sqrt{\lambda_{2 + n}}
\int x^{l-1} P_{1 + n}$ $P_n, ... ,$ $\sqrt{\lambda_{2 + n}}$ $\int x^{l-1} P_{1 + n}$
$P_n )_{n \in \mathbb N},$ where $\left((y_1, \delta_1), ..., (y_{l-1}, \delta_{l-1})
\right) \mapsto$ $  ( \int x$ $G(x)$ $ \frac{dx}{h(x)}, ...,$ $\int x^{l-1}$ $ G(x) $
$ \frac{dx}{h(x)},\int \frac{x F(x)}{2} \frac{dx}{h(x)},...,\int \frac{x^{l-1} F(x)}{2}
\frac{dx}{h(x)} ).$
\end{remark}

\begin{cor}\label{cor4.4}
Suppose that the harmonic measures $1, \omega_1(\infty),...,\omega_{l-1}(\infty)$ are
linearly independent over $\mathbb Q$. Then the following
statements hold:

a) The set of limit points of the associated measures $\{\mu^{(n)}\}_
{n \in \mathbb N},$ see \eqref{a3}, with respect to weak convergence is the set of measures
\begin{equation*}
\begin{split}
  {\mathcal G} :=
  & \{ \frac{h(t)}{2 \pi \prod_{j=1}^{l-1}(t - y_j)} dt + \sum\limits_{j=1}^{l-1}
    \frac{1-\delta_j}{2}
    \frac{\sqrt{H(y_j)}}{\frac{d}{dt} ( \prod_{j = 1}^{l - 1}
    (t - y_j) )_{t = y_j} }\delta(t-y_j) : \\
  & \left. y_j \in [a_{2j}, a_{2 j + 1}], \delta_j \in \{ \pm 1 \} {\rm \ for \ }
    j = 1, ..., l-1  \right\}. \\
\end{split}
\end{equation*}

b) Let ${\mathfrak s}$ be the map associated with coefficient stripping, i.e.
${\mathfrak s} (\mu) = \mu^{(1)}.$ Then ${\mathfrak s}(\mathcal G) \subseteq \mathcal G$
and for every $\mu \in \mathcal G$ the orbit under composition
$\{  {\mathfrak s}^n(\mu) : n \in \mathbb N \}$ is dense in $\mathcal G$ with respect to
weak convergence.
\end{cor}

\begin{proof}
a) By \eqref{a3} on compact subsets of $\mathbb C \setminus [a_1,a_{2l}]$
   $$ \int \frac{1}{z - t} d\mu^{(1+n)} = \sqrt{\lambda_{2+n}} \frac{{\mathcal Q}
      _{1+n}(z)}{{\mathcal Q}_n(z)}
      = \frac{\hat{f}_{(1+n)}(z) - \sqrt{H(z)}}{2 \hat{g}_{(n)}(z)}+ o(1) $$
In the proof of Theorem \ref{thm4.1} we have shown that for each
$\left( (y_1,\delta_1\sqrt{H(y_1)},...,\right.$ $\left. (y_{l-1},\delta_{l-1}\delta_1\sqrt{H(y_{l-1})}) \right)$
there is a sequence $(n_\nu)$ such that
$$ \hat{g}_{(n_\nu)} \underset {\nu \to \infty} \longrightarrow G {\rm \ and \ }
   \hat{f}_{(1 + n_\nu)} \underset {\nu \to \infty} \longrightarrow F ,$$
where $G$ and $F$ are given by \eqref{thmMt1}. Since
\begin{equation*}
\begin{split}
    \frac{\hat{f}_{(1+n)}(z) - \sqrt{H(z)}}{2 \hat{g}_{(n)}(z)}
    & = \frac{1}{2 \pi} \int_E \frac{1}{z - t}
      \frac{h(t)}{\hat{g}_{(n)}(t)} dt \\
    & + \sum \frac{(1-\delta_{j,n})}{2} \frac{\sqrt{H(x_{j,n})}}
      {\hat{g}^{'}_{(n)}(x_{j,n})} \delta(z - x_{j,n})
\end{split}
\end{equation*}
the assertion follows.

b) The invariance with respect to coefficient stripping has been proved in
\cite[Theorem 5]{PehSIAM}. Applying part a) to $\mu \in \mathcal G$ the assertion
follows.
\end{proof}

The set of Jacobi matrices associated with the
set of measures ${\mathcal G}$ is called the isospectral torus nowadays.

\section{Accumulation points of zeros outside the support}

It is well known that polynomials orthogonal with respect to a measure 
$\sigma$ may have zeros in the convex hull of ${\rm supp}(\sigma),$ that is, in the
case under consideration in the gaps $[a_{2j},a_{2j+1}], j \in \{ 1,...,l-1 \}.$ The
same holds for the Weyl solutions ${\mathcal Q}_n.$

\begin{notation}
Let $(n_j)$ be a strictly monotone increasing subsequence of the natural numbers and
let $(f_{n_j})$ be a sequence of functions. We say that a point $y$ is an accumulation
point of zeros of $(f_{n_j})$ if there exists a sequence of points $(y_j)$ such that
$f_{n_j} (y_j) = 0$ and $\lim\limits_j y_j = y.$ As usual, $y$ is called a limit point
of zeros of $(y_j)$ if $U_\varepsilon (y) \setminus \{y \}$, $\varepsilon >0$, contains
an infinite number of $y_j$'s.
\end{notation}

\begin{lemma}\label{HW}
Let $\sigma$ be a positive measure with ${\rm supp}(\sigma)$ bounded and suppose that
$0 < const \leq \lambda_n$ for $n \geq n_0.$ Then the following pairs of sequences
have no common accumulation point of zeros on $\mathbb R \backslash {\rm supp}
(\sigma): $ $(P_{n_j})$ and $(P_{1 + n_j}),$ $({\mathcal Q}_{n_j})$ and
$({\mathcal Q}_{1 + n_j}),$ and $(P_{n_j})$ and $({\mathcal Q}_{n_j}).$
\end{lemma}

\begin{proof}
Since the three sequences $(P_n {\mathcal Q}_n), (P_{n} {\mathcal Q}_{n+1})$
and $(P_{n+1} {\mathcal Q}_n)$ are normal families on
$\mathbb R \backslash {\rm supp}(\mu),$ see \eqref{G1}-\eqref{G3}, they are equicontinuous.
Thus the assertion
follows by the well known relation \\

$ \qquad \qquad \qquad \qquad \qquad P_n {\mathcal Q}_{n+1} -
{\mathcal Q}_n P_{n+1} = -1. $
\end{proof}

By the way that $(P_{n_j})$ and $(P_{1+n_j})$ have no common accumulation point if
$y \notin {\rm supp}(\sigma)$ follows also from
\cite{DenSim}.
Thus $y,$ $y \notin {\rm supp \ \sigma},$ is an accumulation point of zeros of $(P_{n_\nu})$
$(({\mathcal Q}_{n_\nu}))$ if and only if $y$ is a common accumulation point of zeros of
$(P_{n_\nu} {\mathcal Q}_{n_\nu})$ and $(P_{n_\nu} {\mathcal Q}_{1 + n_\nu})$ (of
$(P_{n_\nu} {\mathcal Q}_{n_\nu})$ and $(P_{1 + n_\nu} {\mathcal Q}_{n_\nu})$). In the
following it will be convenient to use this way of expression.

\begin{thm}\label{thm4.2}
Let $\mu$ be given by \eqref{7_2} with $w \in {\rm Sz}(E).$ Then
\begin{equation}\label{thmMt2}
  \begin{split}
      & \lim\limits_\nu \left( \int x P^2_{n_\nu} d\mu, ..., \int x^{l-1} P^2_{n_\nu} d\mu,
        \sqrt{\lambda_{2 + n_\nu}} \int x P_{1 + n_\nu} P_{n_\nu} d\mu, ... , \right.\\
      & \left. \qquad \qquad \qquad \qquad \qquad \qquad \qquad \
        \sqrt{\lambda_{2 + n_\nu}}\int x^{l-1} P_{1 + n_\nu} P_{n_\nu} d\mu \right) \\
    = & \left( \int_E x G(x) \frac{dx}{h(x)}, ..., \int_E x^{l-1} G(x) \frac{dx}{h(x)},
        \int_E \frac{x F(x)}{2} \frac{dx}{h(x)}, ..., \int_E \frac{x^{l-1} F(x)}{2}
        \frac{dx}{h(x)} \right),
  \end{split}
\end{equation}
where $G$ and $F$ are defined in \eqref{thmMt1},
$\left( (y_1, \delta_1), ..., (y_{l-1},\delta_{l-1}) \right) \in $ $\sf{X}_{j=1}^{l-1}$ \newline
$\left( \left((a_{2j}, a_{2j+1})\backslash supp(\mu)\right) \times (\{ \pm 1 \}) \right),$
if and only if for $j = 1,...,l-1$ the point
$y_j,$ $y_j \in (a_{2j}, a_{2j+1}) \backslash supp(\mu),$ is a common
accumulation point of zeros of $(P_{n_\nu} Q_{n_\nu})$ and $(P_{(1+\delta_j)/2 + n_\nu}$
$Q_{(1 - \delta_j)/2 + n_\nu} ),$ $\delta_j \in \{ \pm 1 \}.$
\end{thm}

\begin{proof}
Necessity. In the proof of the sufficiency part of Theorem \ref{thm4.1}a) we have shown that
relation \eqref{thmMt2} implies with the help of Corollary \ref{corR31} that the relations
\eqref{M3x1} hold.
Moreover by \eqref{P3} the $\delta_{j,n_\nu}$'s from \eqref{P2} satisfy $\delta_{j,n_\nu}
\underset {\nu \to \infty} \longrightarrow \delta_j, \ j = 1,...,l-1,$ where $\delta_j
\in \{ -1,1 \},$ since every zero of $G$ lies in the interior of the gaps. Taking a
look at the three limit relations \eqref{t13}, \eqref{a1} and \eqref{a2} in conjunction
with \eqref{P2} the assertion is proved
using Hurwitz's Theorem \cite{HurCou} about the zeros of uniform convergent sequences
of analytic functions.

Sufficiency. First let us note that for any subsequence $(n_\kappa)$ of $(n_\nu)$ for
which the limit in \eqref{t13}, \eqref{a1} and \eqref{a2} exists the limit functions
$\lim _\kappa \hat{g}_{(n_\kappa)} = G$ and $\lim_\kappa \hat{f}_{(1+ n_\kappa)} = F$
are monic polynomials of degree $l-1$ and $l$, respectively, which satisfy, using the
assumption and Hurwitz Theorem, $G(y_j)=0$ and $F(y_j) = \delta_{j}\sqrt{H(y_j)}$
for $j=1,...,l-1$. Thus $G$ and $F$, recall \eqref{P3}, are unique; in other words the
limit of all three sequences $(P_{n_\nu} {\mathcal Q}_{n_\nu})$,
$(P_{n_\nu} {\mathcal Q}_{1 + n_\nu})$
and $(P_{1 + n_\nu} {\mathcal Q}_{n_\nu})$) exists uniformly on $\Omega$ and is given
by $G$ and $F$ which proves by \eqref{G1}-\eqref{G3}, taking into consideration the
last relation from \eqref{M3x2} and \eqref{M3x3} respectively, the sufficiency part.
\end{proof}

Combining Theorem \ref{thm4.2} and Theorem \ref{thm4.1} b) we obtain immediately
(for a wider class of measures) a result of the author \cite[Theorem 3.9]{PehIMNR}
concerning the denseness of zeros of $(P_n)$ in the gaps. For absolutely continuous,
smooth measures the existence of a sequence $(n_\nu)$ such that $(P_{n_\nu})$ has no
zeros in the gaps was shown first in \cite{Sue1}.

\section{Consequences for the recurrence coefficients}

First let us demonstrate how to express Theorem \ref{thm4.2} in terms of limit
relations of the recurrence coefficients and of the accumulation points of zeros
of $(P_n)$ and $({\mathcal Q}_n).$ It is well known that $\int x^j P^2_n$ and
$\sqrt{\lambda_{n+2}}\int x^j P_{n+1}$ $ P_n,$ $j \in \mathbb N,$ can be expressed
in terms of the recurrence coefficients using the recurrence relation of the
$P_n$'s, e.g.
$$ \int x P^2_n = \alpha_{n+1},  \int x^2 P^2_n = \lambda_{n+2} + \alpha_{n+1}^2 +
   \lambda_{n+1}, ...  $$
and
$$ \sqrt{\lambda_{n+2}} \int x P_{n+1} P_n = \lambda_{n + 2},
   \sqrt{\lambda_{n+2}} \int x^2 P_{n+1} P_n = \lambda_{n + 2} (\alpha_{n+2} +
   \alpha_{n+1}), ...  $$
for a closed formula see \cite{Nev}. Solving the system of equations (recall Lemma \ref{lemmaE})
\begin{equation}\label{R1x1}
   \int_E x^j G(x) \frac{d x}{h(x)} = \lim\limits_\nu \int x^j P^2_{n_\nu} \ \
   j = 0,...,l-1,
\end{equation}
respectively,
\begin{equation}\label{R1x2}
   \int_E x^j F(x) \frac{d x}{h(x)} = \lim\limits_\nu 2 \sqrt{\lambda_{2+n_\nu}}
   \int x^j P_{1+n_\nu} P_{n_\nu} \ \ j = 0,...,l-1
\end{equation}
with the help of \eqref{S10} and \eqref{lemmaEx3} we obtain explicit expressions
for the coefficients of $G(x),$ respectively, $F(x)$ in terms of the coefficients
$c_j$ of the series expansion of $1/\sqrt{H(z)},$ see \eqref{S10}, and of the
accumulation points of the recurrence coefficients. This enables us to write
condition \eqref{thmMt2} of Theorem \ref{thm4.2} in terms of accumulation
points of recurrence coefficients and of zeros only. Let us demonstrate this
for the case of two and three intervals.

\begin{cor}\label{cor5.1}
Let $E = [a_1,a_2] \cup [a_3,a_4]$, $d\mu =w(x)dx$ with $w \in {\rm Sz}(E),$ and
let $\delta \in \{ \pm 1 \}.$ Then $y, y \in (a_2,a_3),$ is a common accumulation
point of $(P_{n_\nu} {\mathcal Q}_{n_\nu})$ and $(P_{(1 + \delta)/2 + n_\nu}
{\mathcal Q}_{(1 - \delta)/2 + n_\nu})$ if and only if
\begin{equation}\label{R2x3}
   \lim\limits_\nu \alpha_{1 + n_\nu} = - y + c_1 {\rm \ and \ }
   \lim\limits_\nu \lambda_{2 + n_\nu} = y(c_1 - y) + c_2 - c_1^2 - \delta \sqrt{H(y)}
\end{equation}

Furthermore, for every $(y,\delta) \in (a_2,a_3) \times \{ -1,1 \}$ there exists a
subsequence $(n_\nu)$ of the natural numbers such that \eqref{R2x3} holds, if $E$
is not the inverse image under a polynomial map.

If $E$ is the inverse image under a polynomial map of degree $N$ then for each $b =
0,...,N-1,$ either $(P_{\nu N + b})_{n \in \mathbb N}$ or $({\mathcal Q}_{\nu N + b})_
{n \in \mathbb N}$ has an accumulation point of zeros at the zero $y$ of the $b-$th
associated polynomial $p_{N-1}^{(b)}$ which lies in $(a_2,a_3).$ The limits of
$(\alpha_{1 + \nu N + b})$ and $(\lambda_{2 + \nu N + b})$ are given by \eqref{R2x3},
where $\delta = \mp 1$ if $y$ is an accumulation point of zeros of $(P_{\nu N + b}),$
respectively, of $({\mathcal Q}_{\nu N + b}).$
\end{cor}

\begin{proof}
Solving the system \eqref{R1x1} and \eqref{R1x2} of equations we obtain that
\begin{equation}\label{R2x1}
    G(x) = x + \lim\limits_\nu \alpha_{1 + n_\nu} - c_1 {\rm \ and \ }
    F(x) = x^2 - c_1 x + 2 \lim\limits_\nu \lambda_{2 + n_\nu} + c_1^2 - c_2
\end{equation}

By \eqref{thmMt1} and the statements of Theorem \ref{thm4.2} the assertion follows, if $E$
is not the inverse image.

In the case when $E$ is the inverse image under a polynomial see \cite{Peh}.
\end{proof}

In the case when $E = [a_1,a_2] \cup [a_3,a_4] \cup [a_5,a_6]$ the solution of the
system of equations \eqref{R1x1} and \eqref{R1x2} yields
\begin{equation}\label{R2x2}
\begin{split}
    G(x)
    & = x^2 + (\tilde{\alpha}_{1 + n_\nu} - c_1)x + \tilde{\lambda}_{2 + n_\nu} +
            \tilde{\lambda}_{1 + n_\nu} + \tilde{\alpha}^2_{1 + n_\nu} -
            \tilde{\alpha}_{1 + n_\nu} c_1 + c_1^2 -c_2 \\
    F(x)
    & = x^3 - c_1 x^2 + (2 \tilde{\lambda}_{2 + n_\nu} + c_1^2 - c_2)x + 2
             \tilde{\lambda}_{2 + n_\nu}
            (\tilde{\alpha}_{2 + n_\nu} + \tilde{\alpha}_{1 + n_\nu} - c_1) \\
    & \ \ \ - (c_1^3 - 2 c_1 c_2 + c_3)\\
\end{split}
\end{equation}
where tilde denotes the limits, that is, $\tilde{\alpha}_{1 + n_\nu} :=
\lim\limits_\nu \alpha_{1 + n_\nu},... .$ Equating coefficients with the polynomials
in \eqref{thmMt1} gives easily the condition on the limits of the recurrence
coefficients such that $y_j,$ $y_j \in (a_{2j},a_{2j+1}), j = 1,2,$ are common
accumulation points of zeros of $(P_{n_\nu} {\mathcal Q}_{n_\nu})$ and
$(P_{(1 + \delta_j)/2 + n_\nu}$ ${\mathcal Q}_{(1 + \delta_j)/2 +
n_\nu}), j = 1,2,$ for given $\delta_j \in \{ \pm 1 \}, j = 1,2.$

To obtain a complete picture of the behaviour of the accumulation points of the recurrence
coefficients let us first show that there is a unique correspondence to that ones of the
moments of the Green's functions.

\begin{prop}\label{propP1}
a) Let $m \in \mathbb N,$ $({\bf x}_m, {\bf y}_m) := ( x_{[\frac{m}{2}]+1},...,x_1,...,$
$x_{- [\frac{m-1}{2}] + 1},$ $y_{[\frac{m-1}{2}] + 2},$ $...,y_2,...,y_{- [\frac{m}{2}]+2}),$
where $x_j, y_j \in \mathbb R,$ and put for $k,j \in {\mathbb N}_0,$ $j \geq k,$
\begin{equation}\label{X}
\begin{split}
  I_{j,k} := I_{j,k}({\bf x}_m, {\bf y}_m) \ = \sum\limits_{\substack{-1 \leq k_i \leq 1\\
             i = 1,2,...,j, \\ \sum_{i=1}^j  k_i = k}} z_{0,k_1} z_{k_1,k_1 + k_2} ...
             z_{k_1 +...+ k_{j-1},k_1 +...+ k_j}
\end{split}
\end{equation}
where
\begin{equation*}
z_{k,j} =
\left\{
\begin{split}
\sqrt{y_{j + 1}} & \ k = j - 1\\
x_{j+1}          & \ k = j \\
\sqrt{y_{j + 2}} & \ k = j + 1
\end{split}
\right.
\end{equation*}
Then
\begin{equation*}
\begin{split}
    {\mathcal I}_m:  & \ {\mathbb R}^m \times {\mathbb R}^m_{+} \rightarrow {\mathbb R}^{2 m} \\
                     & ({\bf x}_m, {\bf y}_m) \longmapsto (I_{1,0},\sqrt{y_2}I_{1,1},I_{2,0},
                       \sqrt{y_2}I_{2,1},...,I_{m,0},\sqrt{y_2}I_{m,1})
\end{split}
\end{equation*}
is a continuous one to one map.

b) Let $\sigma$ be a positive measure and suppose that
$(\alpha_{1+n}(d \sigma),\lambda_{2+n}(d \sigma))$ is bounded and
$(\lambda_{2 + n}(d \sigma))_{n \in \mathbb N}$ is bounded away from zero.
Let $m\in \N$ be fixed, and put $(\mbox{\boldmath$\alpha$}^m_{1+n}(d \sigma),$
$\mbox{\boldmath$\lambda$}^m_{2+n}(d \sigma))_{n \in \mathbb N}$ $:=$
$( \alpha_{[\frac{m}{2}]+1+n}(d \sigma),...,$ $\alpha_{1+n}(d \sigma),...,$
$\alpha_{-[\frac{m-1}{2}]+1+n}(d \sigma),$ $\lambda_{[\frac{m-1}{2}]+2+n}(d \sigma),$ $...,$
$\lambda_{2+n}(d \sigma),$ $...,$ $\lambda_{-[\frac{m}{2}]+2+n}(d\sigma) )_{n \in \mathbb N}$. \\
Then for $n > m$
\begin{equation}
\begin{split}
  & {\mathcal I}_m((\mbox{\boldmath$\alpha$}^m_{1+n}(d \sigma),
    \mbox{\boldmath$\lambda$}^m_{2+n}(d \sigma)) =
     (\int x P^2_n d \sigma, \sqrt{\lambda_{2+n}(d \sigma)} \int x P_n P_{1+n} d \sigma, \\
  &   \quad \ \ \qquad \qquad \qquad \qquad ..., \int x^m P^2_n d\sigma,
    \sqrt{\lambda_{2+n}(d \sigma)} \int x^m P_n P_{1+n} d \sigma)
\end{split}
\label{eq-P1}
\end{equation}
and ${\mathcal I}_m$ is a homeomorphism between the set of accumulation points of the two sequences.
\end{prop}

\begin{proof}
a) First let us note that by \eqref{X}
$I_{j,k},$ $k \in \{ 0,1 \},$ $j \in \{ 1,...,m \},$
depends on the variables $({\bf x}_j,{\bf y}_j)$ only and that
\begin{equation}\label{xp2}
I_{m,0}({\bf x}_m,{\bf y}_m) =
\left\{
    \begin{split}
         & y_{ - \frac{m}{2} + 2 } \prod\limits_{i = - \frac{m}{2} + 3}^{1} y_i +
           e_1({\bf x}_{m-1},{\bf y}_{m-1})
         & {\rm if \ } m {\rm \ is \ even} \\
         & x_{ - (\frac{m-1}{2}) + 1 } \prod\limits_{i = - (\frac{m-1}{2}) + 2}^{1} y_i +
           e_2({\bf x}_{m-1},{\bf y}_{m-1})
         & {\rm if \ } m {\rm \ is \ odd} \\
    \end{split}
\right.
\end{equation}
since in \eqref{X} the lowest possible indices $i$ of $x_i, y_i$ appear only for the choice
$(k_1,...,k_m) = (-1,...,-1,1,...,1),$ where $\mp 1$ appears $m/2-$times, respectively for
$(-1,...,-1,0,1,...,1)$ where $\mp 1$ appears $(m-1)/2$ times. By the way, the highest possible
indices are obtained by reversing the order, that is, for the choice $(1,...,1,-1,...,-1),$
respectively, $(1,...,1,0,-1,...,-1)$ which is needed later.

Furthermore for $m > 1$
\begin{equation}\label{xp3}
   \frac{I_{m,1}({\bf x}_m,{\bf y}_m)}{\sqrt{y_2}} =
   \left\{
      \begin{split}
         & x_{ \frac{m}{2} + 1 } \prod\limits_{i = 3}^{\frac{m}{2} + 1} y_i \ +
           \ e_3({\bf x}_{m-1},{\bf y}_{m-1})
         & {\rm if \ } m {\rm \ is \ even} \\
         & y_{ \frac{m-1}{2} + 2 } \prod\limits_{i = 3}^{ \frac{m-1}{2} + 1} y_i \ +
           e_4({\bf x}_{m-1},{\bf y}_{m-1})
         & {\rm if \ } m {\rm \ is \ odd} \\
      \end{split}
   \right.
\end{equation}
since in \eqref{X} the highest possible indices $i$ of $x_i, y_i$ appear only for the choice
$(k_1,...,k_m) = (1,...,1,0,-1,...,-1),$ where $+ 1$ appears $m/2$ and $-1$ appears $(m-2)/2$
times, respectively, for $(1,...,1,-1,...,-1)$ where $\pm 1$ appears $(m \pm 1)/2$ times.
Now we are able to prove the assertion, that is,
\begin{equation}\label{xp33}
   {\mathcal I}_m({\bf x}_m,{\bf y}_m) = {\mathcal I}_m(\tilde{{\bf x}}_m,
   \tilde{{\bf y}}_m) {\rm \ implies \ }
   ({\bf x}_m,{\bf y}_m) = (\tilde{{\bf x}}_m,\tilde{{\bf y}}_m)
\end{equation}
by induction arguments with respect to $m.$ Since $I_{j,k},$ $k \in \{ 0,1 \},$
depends for $j = 0,...,m-1$ on $(\tilde{{\bf x}}_{m-1},\tilde{{\bf y}}_{m-1})$
only it follows by the induction hypothesis that
\begin{equation}\label{xp4}
   ({\bf x}_{m-1},{\bf y}_{m-1}) = (\tilde{{\bf x}}_{m-1},\tilde{{\bf y}}_{m-1})
\end{equation}
Thus it remains to be shown that
\begin{equation}\label{xp5}
    x_{\pm [\frac{m}{2}]+1} = \tilde{x}_{\pm [\frac{m}{2}]+1} {\rm \ and \ }
    y_{\mp [\frac{m}{2}]+2} = \tilde{y}_{\mp [\frac{m}{2}]+2}
\end{equation}
if $m$ is even, respectively odd. Since by \eqref{xp33}
$$ I_{m,0}(\tilde{{\bf x}}_{m},\tilde{{\bf y}}_{m}) = I_{m,0}({\bf x}_{m},{\bf y}_{m})
   {\rm \ and \ } I_{m,1}(\tilde{{\bf x}}_{m},\tilde{{\bf y}}_{m}) = I_{m,1}({\bf x}_{m},
   {\bf y}_{m})  $$
it follows by \eqref{xp2} and \eqref{xp3} in conjunction with \eqref{xp4} that \eqref{xp5}
holds and thus part a) is proved.

b) In \cite{Nev} it has been shown that, $n > j$,
\begin{equation}\label{xp6}
    I_{j,k}(\mbox{\boldmath$\alpha$}^m_{1+n}(d \sigma),
\mbox{\boldmath$\lambda$}^m_{2+n}(d \sigma)) = \int x^j P_{n} P_{k + n} d\sigma
\end{equation}
which gives \eqref{eq-P1}. By the assumptions on the recurrence coefficients it follows that
the set of accumulation points is a compact set contained in ${\mathbb R}^m \times
{\mathbb R}^m_{+}$ and thus its image under ${\mathcal I}_m$ is a compact set contained in the
range of ${\mathcal I}_m,$ which is by \eqref{xp6} and continuity of ${\mathcal I}_m$ equal to
the set of accumulation points of $(\int x P^2_n d \sigma,$ $\sqrt{\lambda_{2+n}} \int x P_n
P_{1 + n},...,$ $\int x^m P^2_n d \sigma,$ $\sqrt{\lambda_{2 + n}} \int x^m P_n P_{1 + n})
_{n \in \mathbb N}.$
\end{proof}

We point out that the starting point of Proposition \ref{propP1}a) was formula \eqref{xp6}
due to Nevai \cite{Nev}. Combining Theorem \ref{thm4.1} a) and Proposition \ref{propP1} b)
we obtain

\begin{thm}\label{5.4new}
Let $\mu$ be given by \eqref{7_2} with $w \in Sz(E)$ and put
$(\mbox{\boldmath$\alpha$}^{l-1}_{1+n}(d \mu),$
$\mbox{\boldmath$\lambda$}^{l-1}_{2+n}(d \mu))_{n \in \mathbb N} := $
$( \alpha_{[\frac{l-1}{2}]+1+n}(d \mu),...,$ $\alpha_{1+n}(d \mu),...,$
$\alpha_{-[\frac{l-2}{2}]+1+n}(d \mu),$ \newline $\lambda_{[\frac{l-2}{2}]+2+n}(d \mu),$
$...,$ $\lambda_{2+n}(d \mu),$ $...,$ $\lambda_{-[\frac{l-1}{2}]+2+n}(d\mu) )_{n \in \mathbb N}$.

a) $(\mbox{\boldmath$\alpha$}^{l-1}_{1+n_\nu}(d \mu),$
$\mbox{\boldmath$\lambda$}^{l-1}_{2+n_\nu}(d \mu))_{\nu \in \mathbb N}$ converges
if and only if
\begin{equation}\label{tilde-1}
\begin{split}
      \left( \left( n_\nu - (l - 1)/2 \right) \omega_k(\infty) +
      \frac{1}{\pi} \int_E \log \left( w(\xi) \right) \right.
    & \frac{\partial \omega_k(\xi)}
      {\partial n_\xi^{+}} d\xi -
      \left. \sum\limits_{j=1}^{m} \omega_k(d_j) \right)_{\nu \in \mathbb N} \\
\end{split}
\end{equation}
converges modulo $1$ for $k = 1,...,l-1$. Furthermore ${\mathcal T} \circ {\mathcal A}
\circ {\tau}^{-1} \circ {\mathcal I}_{l-1}$ is a homeomorphism between the sets of
accumulation points, where $\mathcal A$ is the Abel map from \eqref{ct}, $\tau$ is given
in Theorem \ref{thm4.1} and ${\mathcal T}$ is the map between ${\rm Jac \ } {\mathfrak R}
/{\mathbb R}$ and the real torus $[0,1]^{l-1}$.

b) ${\tau}^{-1} \circ {\mathcal I}_{l-1}$ is a homeomorphism from
the set of accumulation points of $(\mbox{\boldmath$\alpha$}^{l-1}_{1+n}(d \mu),$
$\mbox{\boldmath$\lambda$}^{l-1}_{2+n}(d \mu))_{n \in \mathbb N}$ into
the torus $\sf{X}_{j=1}^{l-1}$ $( [a_{2j}, a_{2j+1}]^+ \cup [a_{2j}, a_{2j+1}]^-)$,
if $1, \omega_1(\infty),...,\omega_{l-1}(\infty)$ are linearly independent over
$\mathbb Q.$
\end{thm}

\begin{proof}
a) By Proposition \ref{propP1}b) and Theorem \ref{thm4.1} a) the sequence \newline
$( \mbox{\boldmath$\alpha$}^{l-1}_{1 + n_\nu},
\mbox{\boldmath$\lambda$}^{l-1}_{2 + n_\nu} )_{\nu \in \mathbb N}$
converges if and only if $( {\mathfrak x}_{1,n_\nu},...,
{\mathfrak x}_{l-1,n_\nu})_{\nu \in \mathbb N}$ converges,
where $( {\mathfrak x}_{1,n_\nu},...,{\mathfrak x}_{l-1,n_\nu})$
is given by \eqref{nw}. Next let us recall that, by the form of ${\rm Jac \ } {\mathfrak R}/ {\mathbb R}$ and the bijectivity
of the Abel map $\mathcal A$, for every $( {\mathfrak x}_{1},...,{\mathfrak x}_{l-1}) \in \sf{X}_{j=1}^{l-1}$ $( [a_{2j}, a_{2j+1}]^+ \cup [a_{2j}, a_{2j+1}]^-)$ there is an unique $(t_1,...,t_{l-1}) \in [-1/2,1/2]^{l-1}$ such that
\begin{equation}\label{TT}
{\mathcal A}( {\mathfrak x}_{1},...,
{\mathfrak x}_{l-1}) = (B_{i j}) \left((t_1,...,t_{l-1})^t + (m_1,...,m_{l-1})^t \right)
\end{equation}
where $m_\kappa \in \Z$.
Using \eqref{ep6t1} we obtain

$$ \sum_{\kappa = 1}^{l-1} \left( \frac{1}{2} \sum_{j=1}^{l-1} \delta_{j,n_\nu}
\omega_\kappa(x_{j,n_\nu}) \right) B_{k \kappa} =
\sum_{j=1}^{l-1} \frac{1}{2} \int_{\mathfrak x _{j,n_\nu}^*}^{\mathfrak x _{j,n_\nu}}  \varphi_k =
\sum_{\kappa = 1}^{l-1} t_{\kappa,n_\nu} B_{k \kappa}$$
with $t_{\kappa,n_\nu} \in [-1/2,1/2]$ for $k = 1,...,l-1,$ hence
\begin{equation}
\label{f1}
 t_{\kappa,n_\nu} = \frac{1}{2} \sum_{j=1}^{l-1} \delta_{j,n_\nu} \omega_{\kappa}(x_{j,n_\nu}) {\rm \ modulo \ } 1
 \end{equation}
Since, by \eqref{TT}, $( {\mathfrak x}_{1,n_\nu},...,
{\mathfrak x}_{l-1,n_\nu})_{\nu \in \mathbb N}$ converges if and only if $(t_{1,n_\nu},...,t_{l-1,n_\nu})_{\nu \in \mathbb N}$
converges modulo $1$ the assertion follows by \eqref{t9} and \eqref{f1}.

b) Follows immediately by Theorem \ref{thm4.1} b) and Proposition \ref{propP1} b).
\end{proof}

Moreover under the assumptions of Theorem \ref{5.4new} b) the number
of accumulation points of the recurrence coefficients is infinite,
which has been proved for the isospectral torus in \cite{Luk-Peh} already,
see also \cite{Luk}.

\begin{cor}
Let $(n_\nu)$ be such that $({\boldsymbol \alpha}_{1+ n_\nu}^{l - 1}(d \mu),
{\boldsymbol \lambda}_{2 + n_\nu}^{l - 1}(d \mu))_{\nu \in \mathbb N}$
converges. Then the sequence $(n_{\nu + 1} - n_\nu)_{\nu \in \mathbb N}$
is unbounded if $1, \omega_1(\infty),...,\omega_{l - 1}(\infty)$ are linearly
independent over $\mathbb Q.$
\end{cor}

\begin{proof}
By \eqref{t7} and Theorem \ref{5.4new} a) we get that the
RHS of the system of equations $k = 1,..., l - 1,$
\begin{equation}
\begin{split}
    & (n_{\nu + 1} - n_{\nu}) \omega_k(\infty) - 2 (m_{k,n_{\nu + 1}} - m_{k,n_\nu}) = \\
  = & \frac{1}{2} \sum_{j=1}^{l - 1} (\delta_{j,n_{\nu+1}} \omega_k(x_{j,n_{\nu + 1}}) -
                                      \delta_{j,n_{\nu}} \omega_k(x_{j,n_{\nu}} ) )
\end{split}
\end{equation}
tends to zero, where $m_k, n_{\nu + 1}, m_k, n_\nu \in \mathbb Z.$ Assume that
$(n_{\nu + 1} - n_\nu)$ is bounded and thus that the LHS takes on modulo $1$
a finite number of values only and is not equal to zero, since $\omega_k(\infty)$
must not be rational, which yields the desired contraction.
\end{proof}

In particular Theorem \ref{5.4new} holds true for measures from
the isospectral torus and thus holds true even for measures,
whose recurrence coefficients satisfy \eqref{F1}, which are described
in \cite{PehYud}. Therefore it looks like we could have restricted
to the isospectral torus. The problem is that for accumulation
points of zeros which is the other main point of our studies
there are no corresponding statements like \eqref{F1} available.

Thus the sequences $(n_\nu)$ for which the recurrence coefficients
$(\mbox{\boldmath$\alpha$}^{l-1}_{1+n_\nu}(d \mu),$
$\mbox{\boldmath$\lambda$}^{l-1}_{2+n_\nu}(d \mu))_{\nu \in \mathbb N}$
converge are determined by the harmonic measures $E$, only the values of
the accumulation points depend on $\mu$. We believe that the property
"$(\mbox{\boldmath$\alpha$}^{l-1}_{1+n_\nu}(d \sigma),$
$\mbox{\boldmath$\lambda$}^{l-1}_{2+n_\nu}(d \sigma))_{\nu \in \mathbb N}$
converges if and only if $(n_\nu \mbox{\boldmath$\omega$}(\infty)) _{\nu \in \mathbb N}$
converges modulo $1$" holds for a wide class of measures $\sigma$
whose essential support is $E$; loosely speaking that it is the
counterpart of the measures
whose essential support is a single interval and for which the recurrence
coefficients converge, for instance as in Rakhmanov's \cite{Rah2} and Denisov's \cite{Den} case.

As an immediate consequence of Thm. \ref{5.4new} and Prop. \ref{propP1}
in conjunction with Cor. \ref{corR31} we obtain
 \begin{cor} The Green's functions $(G(z,n_\nu,n_\nu))_{\nu \in \mathbb N}$
 and $(G(z,1+n_\nu,$ $n_\nu))_{\nu \in \mathbb N}$, defined in \eqref{G1} and
 \eqref{G2}, converge simultaneously uniformly on compact subsets of $\Omega $
 if and only if $(n_\nu \mbox{\boldmath$\omega$}(\infty)) _{\nu \in \mathbb N}$
 converges modulo $1$.
 \end{cor}

\begin{thm}\label{5.4} The following statements hold for the recurrence coefficients
of any measure $\mu$ of the form \eqref{7_2} with $w \in Sz(E)$. a) For every $k\in \Z$
$\lim\limits_\nu (\mbox{\boldmath$\alpha$}^{l-1}_{k+1+n_\nu},
\mbox{\boldmath$\lambda$}^{l-1}_{k+2+n_\nu})$ exists if $\lim\limits_\nu
(\mbox{\boldmath$\alpha$}^{l-1}_{1+n_\nu},
\mbox{\boldmath$\lambda$}^{l-1}_{2+n_\nu})$ exists.

b)Put $(\tilde{\mbox{\boldmath$\alpha$}}^{l-1}_{k+1+(n_\nu)},
\tilde{\mbox{\boldmath$\lambda$}}^{l-1}_{k+2+(n_\nu)}) :=$
$\lim\limits_\nu (\mbox{\boldmath$\alpha$}^{l-1}_{k+1+n_\nu},
\mbox{\boldmath$\lambda$}^{l-1}_{k+2+n_\nu})$. The limits
are related to each other by
$$(\tilde{\mbox{\boldmath$\alpha$}}^{l-1}_{k+1+(n_\nu)},
\tilde{\mbox{\boldmath$\lambda$}}^{l-1}_{k+2+(n_\nu)}) =
  \psi^k ((\tilde{\mbox{\boldmath$\alpha$}}^{l-1}_{1+(n_\nu)},
  \tilde{\mbox{\boldmath$\lambda$}}^{l-1}_{2+(n_\nu)}))$$
where $\psi$ is a continuous map on $\R^{l-1}\times \R_+^{l-1}$ which does not depend on
$\mu$ and $\psi^k$ denotes the $k$-th composition. (Note that $\eta^{-1} \circ \psi^k
\circ \eta,$ where $\eta = \mathcal{I}_{l-1}^{-1}\circ \tau ,$ maps the corresponding
points on the torus to each other). \\
c) The orbit $\{ \psi^k ((\tilde{\mbox{\boldmath$\alpha$}}^{l-1}_{1+(n_\nu)},
  \tilde{\mbox{\boldmath$\lambda$}}^{l-1}_{2+(n_\nu)})) : k \in {\mathbb N}_0 \}$ is dense
in the set of accumulation points of the sequence
$(\mbox{\boldmath$\alpha$}^{l-1}_{1 + n},
\mbox{\boldmath$\lambda$}^{l-1}_{2 + n})_{n \in \mathbb N}$ if
$1, \omega_1(\infty),...,\omega_{l-1}(\infty)$ are linearly independent over $\mathbb Q.$
\end{thm}

\begin{proof} Part a) follows by Proposition \ref{propP1}
b) and the fact that by \eqref{R31x1} and \eqref{R31x2} the limits
$\lim\limits_\nu \int x^j P^2_{n_\nu}$ and $\lim\limits_\nu \sqrt{\lambda_{2+n_\nu}} \int
x^j P_{n_\nu} P_{1+n_\nu}$ exist for every $j \in \mathbb N$ if they exist for $j=0,...,l-1$.

b) Let us first show the assertion for $k=1.$ Taking a look at
$(\tilde{\mbox{\boldmath$\alpha$}}^{l-1}_{2+(n_\nu)},\tilde{\mbox{\boldmath$\lambda$}}
^{l-1}_{3+(n_\nu)})$
we see that we have to show only that
\begin{equation}\label{tilde1}
\begin{split}
    & \lim\limits_{\nu} \alpha_{[\frac{l-1}{2}] + 2 + n_\nu} =
      f(\tilde{\mbox{\boldmath$\alpha$}}^{l-1}_{1+(n_\nu)},
      \tilde{\mbox{\boldmath$\lambda$}}^{l-1}_{2+(n_\nu)}) {\rm \ and \ }\\
    & \lim\limits_{\nu} \lambda_{[\frac{l-1}{2}] + 3 + n_\nu} =
      g(\tilde{\mbox{\boldmath$\alpha$}}^{l-1}_{1+(n_\nu)},
      \tilde{\mbox{\boldmath$\lambda$}}^{l-1}_{2+(n_\nu)}),
\end{split}
\end{equation}
where $f,g$ are continuous functions. By \eqref{R31x2} and Proposition \ref{propP1} b)
$$ \lim\limits_{\nu} \sqrt{\lambda_{2 + n_\nu}} \int x^l P_{1+n_\nu} P_{n_\nu} =
   h_1(\tilde{\mbox{\boldmath$\alpha$}}^{l-1}_{1+(n_\nu)},
   \tilde{\mbox{\boldmath$\lambda$}}^{l-1}_{2+(n_\nu)})   $$
and thus by \eqref{xp3}, recall \eqref{xp6},
\begin{equation}\label{tilde2}
\begin{split}
   & \lim\limits_\nu \alpha_{\frac{l}{2} + 1 + n_\nu} =
     f_1(\tilde{\mbox{\boldmath$\alpha$}}^{l-1}_{1+(n_\nu)},
     \tilde{\mbox{\boldmath$\lambda$}}^{l-1}_{2+(n_\nu)}),
     {\rm \ respectively}, \\
   & \lim\limits_\nu \lambda_{[\frac{l-1}{2}] + 2 + n_\nu} =
     g_1(\tilde{\mbox{\boldmath$\alpha$}}^{l-1}_{1+(n_\nu)},
     \tilde{\mbox{\boldmath$\lambda$}}^{l-1}_{2+(n_\nu)}),
\end{split}
\end{equation}
if $l$ is even, respectively, $l$ is odd.

Next let us observe that by \eqref{X} and \eqref{xp6} for $n \geq l$
\begin{equation}\label{tilde3}
\int x^l P^2_{1+n} =
\left\{
\begin{split}
   & \lambda_{\frac{l}{2} + 2 + n} \prod\limits_{i = n + 3}^{n + 1 + \frac{l}{2}} \lambda_i +
     \tilde{e}_1 ( \alpha_{\frac{l}{2} + 1 + n}, \mbox{\boldmath$\alpha$}^{l-1}_{1 + n},
     \mbox{\boldmath$\lambda$}^{l-1}_{2 + n} )
   & {\rm if \ } l {\rm \ is \ even} \\
   & \alpha_{(\frac{l-1}{2}) + 2 + n} \prod\limits_{i = n + 3}^{n + 2 + (\frac{l-1}{2})} \lambda_i +
     \tilde{e}_2 ( \lambda_{\frac{l-1}{2} + 2 + n}, \mbox{\boldmath$\alpha$}^{l-1}_{1 + n},
     \mbox{\boldmath$\lambda$}^{l-1}_{2 + n} )
   & {\rm if \ } l {\rm \ is \ odd}, \\
\end{split}
\right.
\end{equation}
where $\tilde{e}_1, \tilde{e}_2$ are continuous functions. By \eqref{R31x2}, \eqref{tilde3} applied to
$l-1,$ and Proposition \ref{propP1} b) it follows that
\begin{equation*}
\lim\limits_\nu \int x^l P^2_{1 + n_\nu} =
\left\{
\begin{split}
   & h_2 ( \tilde{\alpha}_{\frac{l-2}{2} + 2 + (n_\nu) },
     \tilde{\mbox{\boldmath$\alpha$}}^{l-1}_{1 + (n_\nu)},
     \tilde{\mbox{\boldmath$\lambda$}}^{l-1}_{2 + (n_\nu)} )
   & {\rm if \ } l {\rm \ is \ even} \\
   & h_3 ( \tilde{\lambda}_{\frac{l-1}{2} + 2 + (n_\nu)},
     \tilde{\mbox{\boldmath$\alpha$}}^{l-1}_{1 + (n_\nu)},
     \tilde{\mbox{\boldmath$\lambda$}}^{l-1}_{2 + (n_\nu)} )
   & {\rm if \ } l {\rm \ is \ odd} \\
\end{split}
\right.
\end{equation*}
Thus by \eqref{tilde3} and \eqref{tilde2} we obtain that
\begin{equation*}
\begin{split}
    & \lim\limits_{\nu} \lambda_{\frac{l}{2} + 2 + n_\nu} =
      g_2(\tilde{\mbox{\boldmath$\alpha$}}^{l-1}_{1+(n_\nu)},
      \tilde{\mbox{\boldmath$\lambda$}}^{l-1}_{2+(n_\nu)}),
      {\rm \ respectively, \ } \\
    & \lim\limits_{\nu} \alpha_{\frac{l-1}{2} + 2 + n_\nu} =
      f_2(\tilde{\mbox{\boldmath$\alpha$}}^{l-1}_{1+(n_\nu)},
      \tilde{\mbox{\boldmath$\lambda$}}^{l-1}_{2+(n_\nu)})
\end{split}
\end{equation*}
if $l$ is even, respectively, $l$ is odd, and the relations \eqref{tilde1} are proved.

For arbitrary $k$ the statement follows immediately by iteration using
induction arguments and the fact that the relations \eqref{tilde1} hold
for any $(n_\nu)$ for which we have convergence.

c) Since $ (n_\nu \mbox{\boldmath$\omega$}(\infty) -
\mbox{\boldmath$c$}) \underset{\nu \to \infty} \longrightarrow
\mbox{\boldmath$\gamma$}\quad \text{modulo}\; 1$ implies that
$(n_\nu +k)\mbox{\boldmath$\omega$}(\infty) -
\mbox{\boldmath$c$} \underset{\nu \to \infty} \longrightarrow \mbox{\boldmath$\gamma$}
+ k\mbox{\boldmath$\omega$}(\infty)\quad \text{modulo}\; 1$
the assertion follows by part b) and Thm. \ref{5.4new} using the fact that the RHS is
dense in $[0,1]^{l-1}$ with respect to $k$,
$k\in \Z$.
\end{proof}

The map $\psi$ can be obtained explicitly by Corollary \ref{corR31} and \eqref{xp6},
at least for small $l$. Most likely Theorem \ref{5.4} c) holds true without the assumption of linear
independence of $1, \omega_1(\infty),...,\omega_{l-1}(\infty).$
Note that Theorem \ref{5.4} is of so-called "oracle type", see \cite{Rem}.
Finally let us remark that we expect (investigations are on the way) that
Theorem \ref{thm4.2} and Theorem \ref{5.4new} hold true for so-called
homogeneous sets $E$ and a possible infinite set of mass points lying outside
$E$ and accumulating on $E$.

\end{document}